\documentclass[english,12pt,oneside]{amsproc}
\usepackage[english]{babel}
\usepackage{a4wide}
\usepackage{amsthm}
\usepackage{graphics}
\usepackage{amsfonts, amssymb, amscd, amsmath}
\usepackage{latexsym}
\usepackage[matrix,arrow,curve]{xy}
\usepackage{mathabx}
\usepackage{color}
\usepackage{pbox}
\usepackage{tikz}
\usetikzlibrary{matrix,decorations.pathreplacing,positioning}
\usepackage{hyperref}

\DeclareMathOperator{\Cone}{Cone} \DeclareMathOperator{\pt}{pt}
\DeclareMathOperator{\Ker}{Ker}

\DeclareMathOperator{\const}{const}
\DeclareMathOperator{\Hom}{Hom} \DeclareMathOperator{\rk}{rk}
\DeclareMathOperator{\diag}{diag}
\DeclareMathOperator{\Pe}{Pe} \DeclareMathOperator{\Fl}{Fl}
\DeclareMathOperator{\tr}{tr} \DeclareMathOperator{\re}{Re}

\DeclareMathOperator{\Cy}{Cy} \DeclareMathOperator{\St}{St}
\DeclareMathOperator{\Arg}{Arg} \DeclareMathOperator{\dist}{dist}
\DeclareMathOperator{\free}{free}
\DeclareMathOperator{\Hilb}{Hilb}

\newcommand{\edt}[1]{\textcolor{black}{#1}}
\newcommand{\edtt}[1]{\textcolor{black}{#1}}

\newcommand{\Sch}{Schr\"{o}dinger }

\newcommand{\minel}{\hat{0}}

\newcommand{\ko}{\Bbbk}
\newcommand{\Zo}{\mathbb{Z}}
\newcommand{\Ro}{\mathbb{R}}
\newcommand{\Rg}{\mathbb{R}_{\geqslant 0}}
\newcommand{\Co}{\mathbb{C}}

\newcommand{\br}{\widetilde{\beta}}

\newcommand{\wh}[1]{{\widehat{#1}}}

\newcommand{\stir}[2]{{#1\atopwithdelims\{\}#2}}

\newcommand{\Hr}{\widetilde{H}}
\newcommand{\dd}{\partial}

\newcommand{\Pa}{\mathcal{P}}
\newcommand{\T}{\mathcal{T}}

\newcommand{\I}{\mathbb{I}}
\newcommand{\po}{\tilde{p}}

\newcommand{\B}{\mathbb{B}}
\newcommand{\A}{\mathcal{A}}
\newcommand{\Xn}{X_{n,\lambda}}
\newcommand{\Xneps}{X_{n,\lambda}^{\leqslant\varepsilon}}
\newcommand{\Xngeps}{X_{n,\lambda}^{\geqslant\varepsilon}}
\newcommand{\eps}{\varepsilon}

\newcommand{\Qn}{Q_{n,\lambda}}

\newcommand{\Qngeps}{Q_{n,\lambda}^{\geqslant\varepsilon}}

\newcommand{\PT}{\mathcal{PT}}
\newcommand{\KPT}{K_{\mathcal{PT}}}

\newcommand{\CP}{\mathbb{C}P}

\newcounter{stmcounter}[section]

\numberwithin{equation}{section}

\theoremstyle{plain}
\newtheorem{cor}[stmcounter]{Corollary}

\newtheorem{thm}[stmcounter]{Theorem}

\newtheorem{prop}[stmcounter]{Proposition}
\newtheorem{lem}[stmcounter]{Lemma}

\theoremstyle{definition}
\newtheorem{defin}[stmcounter]{Definition}

\theoremstyle{remark}
\newtheorem{ex}[stmcounter]{Example}
\newtheorem{rem}[stmcounter]{Remark}
\newtheorem{con}[stmcounter]{Construction}

\begin{document}

\title{Space of isospectral periodic tridiagonal matrices}

\author{Anton Ayzenberg}
\address{Faculty of computer science, Higher School of Economics}
\email{ayzenberga@gmail.com}
%
\thanks{
This work is supported by the Russian Science Foundation under grant 18-71-00009.
}
\begin{abstract}
A periodic tridiagonal matrix is a tridiagonal matrix with
additional two entries at the corners. We study the space $X_{n,\lambda}$ of
Hermitian periodic tridiagonal $n\times n$-matrices with a fixed
simple spectrum $\lambda$. Using the discretized S\edt{c}hr\"{o}dinger operator we
describe all spectra $\lambda$ for which $X_{n,\lambda}$ is a topological
manifold. The space $X_{n,\lambda}$ carries a natural effective action of a
compact $(n-1)$-torus. We describe the topology of its orbit space
and, in particular, show that whenever the isospectral space is a
manifold, its orbit space is homeomorphic to $S^4\times T^{n-3}$.
There is a classical dynamical system: the flow of the periodic Toda lattice, acting
on $X_{n,\lambda}$. Except for the degenerate locus $X_{n,\lambda}^0$, the Toda lattice
exhibits Liouville--Arnold behavior, so that the space
$X_{n,\lambda}\setminus X_{n,\lambda}^0$ is fibered into tori.
The degenerate locus of the Toda system is described in terms of combinatorial geometry:
its structure is encoded in the special cell subdivision of a torus, which is obtained
from the regular tiling of the euclidean space by permutohedra. We apply methods of
commutative algebra and toric topology to describe the cohomology and equivariant
cohomology modules of $X_{n,\lambda}$.
\end{abstract}

\maketitle

\section{Introduction}

Let $\Gamma=(V,E)$ be a simple graph on a set
$V=[n]=\{1,\ldots,n\}$. Let $M_\Gamma$ be the vector space of
Hermitian $n\times n$-matrices $A=(a_{ij})$, such that $a_{ij}=0$
for $(i,j)\notin E$. We consider the space
$M_{\Gamma,\lambda}\subset M_\Gamma$ of all such matrices, which have a
given simple spectrum
$\lambda=(\lambda_1<\lambda_2<\cdots<\lambda_n)$. Note that each
space $M_{\Gamma,\lambda}$ carries the conjugation action of a
compact torus $T^n$. The action is noneffective: scalar matrices
commute with every matrix, hence the diagonal subgroup of $T^n$
acts trivially.

Several examples are well studied. The complete graph $\Gamma=K_n$
corresponds to the space of all isospectral matrices, which is
diffeomorphic to the variety $\Fl_n$ of complete flags in $\Co^n$.
The path graph $\Gamma=\I_n$ with $n+1$ vertices produces the
space $M_{\I_n,\lambda}$ of isospectral tridiagonal matrices,
which is known to be a smooth $2n$-manifold; its smooth type
is independent of $\lambda$. The real version of
$M_{\I_n,\lambda}$ is called the Tomei manifold: it was introduced
and studied in \cite{Tomei}. The $T^n$-action on
$M_{\I_n,\lambda}$ is locally standard and its orbit space is
diffeomorphic to a simple polytope, the permutohedron
\cite{Tomei,DJ}. Note that $M_{\I_n,\lambda}$ is not a toric
variety, although it is closely related to the permutohedral variety \cite{BFR}.

More generally, the spaces $M_{\Gamma_h,\lambda}$ corresponding to
indifferent graphs $\Gamma_h$ are the spaces of staircase
matrices. It is more convenient to encode this type of spaces by
Hessenberg functions. The Hessenberg function is a function
$h\colon[n]\to[n]$ such that $h(i)\geqslant i$ and
$h(i+1)\geqslant h(i)$. The space $M_{\Gamma_h}$ is the space of
Hermitian matrices $A$ such that $a_{ij}=0$ for $j>h(i)$. Every
space $M_{\Gamma_h,\lambda}$ is a smooth manifold independent of a
simple spectrum $\lambda$. Its odd degree cohomology modules vanish,
therefore $M_{\Gamma_h,\lambda}$ is equivariantly formal
in the sense of Goresky--\edt{Kottwitz--}MacPherson (see Definition \ref{definEqForm}). The equivariant
cohomology ring of $M_{\Gamma_h,\lambda}$
can be described using GKM-theory \cite{GKM,Kur}. See \cite{ABhess} for details on the
the spaces $M_{\Gamma_h,\lambda}$ and their relation to regular semi-simple Hessenberg
varieties.

For the star graph $\Gamma=\St_n$ (see Fig.\ref{figGraphs}), the
space $M_{\St_n,\lambda}$ is also a smooth manifold, and its diffeomorphism type
does not depend on $\lambda$. The effective action of $T=T^{n+1}/\Delta(T^1)$ on
$M_{\St_n,\lambda}$ is locally standard, therefore the orbit space
$Q_{\St_n,\lambda}=M_{\St_n,\lambda}/T$ is a manifold with
corners. Unlike the case of tridiagonal matrices, the orbit space
$Q_{\St_n,\lambda}$ for $n\geqslant 3$ is not a simple polytope.
The topology of $Q_{\St_n,\lambda}$ itself is quite complicated,
and it is difficult to state any general result about the manifold
$M_{\St_n,\lambda}$ itself. However, the topology can be described
in details for $n=4$, which was done in \cite{ABarrow}.

\begin{figure}[h]
\begin{center}
\includegraphics[scale=0.45]{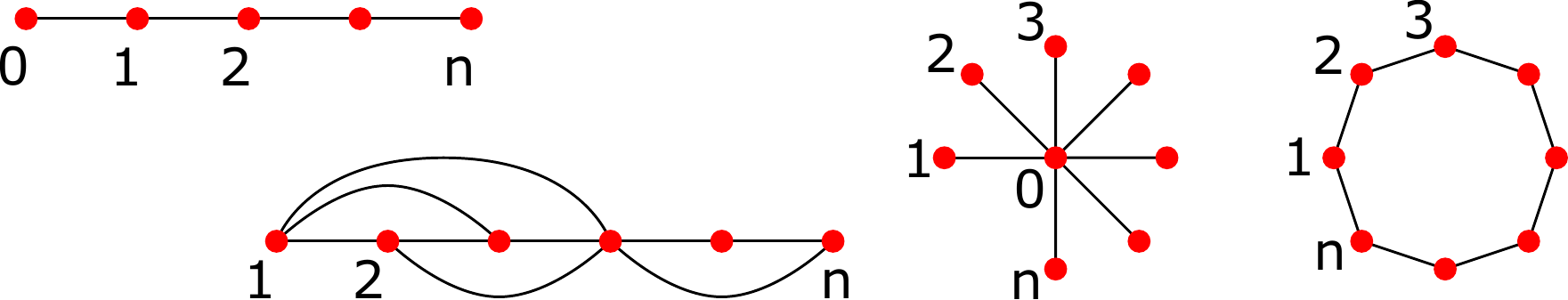}
\end{center}
\caption{Particular graphs, encoding important isospectral matrix
spaces: the path graph $\I_n$, indifferent graphs $\Gamma_h$, the
star graph $\St_n$, and the cycle graph $\Cy_n$}\label{figGraphs}
\end{figure}

In this paper we consider the case $\Gamma=\Cy_n$, the cyclic
graph on $n$ vertices. The Hermitian matrices corresponding to
$\Cy_n$ have the form
\begin{equation}\label{eqCyclicMatrix}
L=L(\underline{a},\underline{b})=\begin{pmatrix}
a_1 & b_1& 0 & \cdots & \overline{b}_n\\
\overline{b}_1& a_2 & b_2 & 0 & \vdots\\
0 & \overline{b}_2 & a_3 & \ddots & 0\\
\vdots & 0 & \ddots &\ddots&  b_{n-1}\\
b_n& \cdots & 0 & \overline{b}_{n-1} &a_n
\end{pmatrix},
\end{equation}
where $a_i\in \Ro$, $b_i\in \Co$. Such matrices are called \emph{periodic tridiagonal matrices} or
\emph{periodic Jacobi matrices}. We will simply call them
\emph{periodic}. It is assumed throughout the paper that $n\geqslant 3$.

The space $\Xn=M_{\Cy_n,\lambda}$ of all periodic matrices with a
simple spectrum $\lambda$ has dimension $2n$, and carries an
effective action of $T=T^{n-1}$. Hence the torus action has
complexity one. The difference between half the real
dimension of a manifold and the dimension of a torus is called \emph{the
complexity of the action}: this terminology naturally comes from both
algebraic geometry and symplectic geometry.

We prove that under certain conditions on a simple spectrum, the
space $\Xn$ is not a smooth manifold, not even a homology
manifold, see Theorem \ref{thmNonSmooth}. This gives a negative
answer to our question, posed in \cite{ABarrow}. This also settles
certain inaccuracy appearing in the work of van Moerbeke
\cite{VanM}, who studied the real analogue of $\Xn$.

For any simple spectrum $\lambda$, we describe the topology of
the orbit space $\Xn/T$, see Corollary \ref{corOrbitSpaceGeneral}.
If $\Xn$ is a topological manifold, we prove that the orbit space $\Xn/T$ is
homeomorphic to the product $S^4\times T^{n-3}$. When $n=3$, the
space $X_{3,\lambda}$ is the space of all Hermitian matrices with
the given spectrum $\lambda$. This space is diffeomorphic to the full complex
flag variety $\Fl_3$. Hence, for $n=3$, we recover the result of
Buchstaber--Terzic \cite{BTober,BT,BT2}, which states that
$\Fl_3/T^2\cong S^4$. Note that the action is not free, however
the orbit space is still a topological manifold. This fact is
consistent with the general theory developed in~\cite{AyLoc}.

The main ingredient of our arguments is the product of
off-diagonal elements
\[
B=\prod\nolimits_{i=1}^nb_i\in \Co
\]
of the periodic matrix $L(\underline{a},\underline{b})$ \edt{and the corresponding map $p\colon \Xn\to \Co$, $p(L(\underline{a},\underline{b}))=B$}. We show
that with the matrix spectrum fixed, the number $B$ takes values
inside a compact convex subset $\B\subset\Co$, lying between two
confocal parabolas, see Theorem \ref{thmImage}. This statement may
be considered a folklore: its real version was proved in
\cite{VanM,Krich}, and the complex version is not more
complicated. In Section \ref{secSpectralCurve} we briefly review
the necessary facts about discrete \Sch operator, needed for
this result.

The value $B$ is preserved by the torus action, hence there is a
map $\po\colon\Xn/T\to\Co$ from the orbit space, evaluating the
number $B$. The set $\po^{-1}(\Co\setminus\{0\})$ consists of free
orbits. However the torus action has nontrivial $T$-equivariant skeleton,
which is a proper subset of $\po^{-1}(0)$. To describe the structure of the
equivariant skeleton, we use combinatorial geometry.

It is well known that euclidean space can be tiled by parallel
copies of a regular permutohedron. Taking quotients by lattices in
a euclidean space, we may produce many interesting permutohedral
cell subdivisions of a torus. We show that a certain lattice
produces a regular cell subdivision $\PT^{n-1}$ of an
$(n-1)$-dimensional torus, which we called \emph{the wonderful
subdivision}. It has several interesting properties. First, it
models the equivariant skeleton of the torus action on $\Xn$.
Second, this wonderful subdivision minimizes the number of facets
among all possible regular cell subdivisions of a torus. Such subdivisions
and their dual simplicial cell subdivisions for general PL-manifolds
are known in combinatorial topology under the name of crystallizations \cite{FGG}.
We briefly recall the required combinatorial geometry in Section~\ref{secTilings}.

Next we describe the topology of the whole space $\Xn$.
Let $\Xn^0=\edt{p}^{-1}(0)$ denote the subset of matrices with $B=0$.
The space $\Xn$ is smooth in vicinity of $\Xn^0$: this actually
follows from the properties of non-periodic Toda lattice, see Proposition
\ref{propSmoothOverZero}. Using the
result of \cite{AyLoc} concerning the topological classification of complexity one
torus actions, we describe the topology of a
small neighborhood $\Xneps$ of $\Xn^0$. It happens that, up to
homeomorphism, the $T^{n-1}$-action on $\Xneps$ can be extended to
a locally standard $T^n$-action on this space. The necessary
notions related to complexity one torus actions are given in
Section \ref{secNeighborhoodGeneral}.

In a series of works \cite{Ay1,Ay2,Ay3,AMPZ} we developed a
toolbox to compute cohomology and equivariant cohomology of
manifolds with locally standard torus \edt{action} whose orbit spaces
have acyclic proper faces. This toolbox is
applied to the subspace $\Xneps$. The $T^n$-orbit space of $\Xneps$ is a
manifold with corners, whose face structure is the wonderful cell
subdivision of a torus, hence all its proper faces are acyclic, so
we are in position to apply the general technique. The
algebro-topological invariants of $\Xneps$ are computed in terms
of combinatorial invariants of the wonderful cell subdivision
$\PT^{n-1}$. We recall the theory of $h$-, $h'$-, and
$h''$-numbers of simplicial posets and compute these invariants
for the dual simplicial poset of the wonderful subdivision in
Section \ref{secNeighborhoodCombinatorics}.

In Section \ref{secEquivCohom} we describe the additive
structure of $T^{n-1}$-equivariant cohomology modules
of the neighborhood $\Xneps$. The ordinary Betti numbers of $\Xn$ are calculated in
Section \ref{secBetti}. There we also prove that $\Xn$ is not equivariantly formal
for $n\geqslant 4$ by comparing equivariant and ordinary Betti numbers of $\Xn$.

\section{Torus action and Toda flow}\label{secActionAndFlow}

The element $t=(t_1,\ldots,t_n)\in T^n$ acts on a cyclic matrix by
the formula
\begin{equation}\label{eqTorusActionCoords}
tL(\underline{a};b_1,\ldots,b_{n-1},b_n)=L(\underline{a};t_1t_2^{-1}\cdot
b_1,\ldots,t_{n-1}t_n^{-1}\cdot b_{n-1},t_nt_1^{-1}\cdot b_n).
\end{equation}
It is easy to see that the torus action preserves the quantity
$B=\prod_1^nb_i$. The action is non-effective: the scalar matrices act
trivially. Hence we consider the effective action of the
quotient torus $T^{n-1}=T^n/\Delta(S^1)$ on $\Xn$.

Apart from the torus action, there is a classical dynamical system acting on the space of periodic
matrices: the \emph{periodic Toda lattice}. We now briefly recall the definition and properties
of this dynamical system.

\begin{con}
For a matrix $L=L(\underline{a},\underline{b})$ consider the
skew-Hermitian matrix
\[
P=P(L)=\begin{pmatrix}
0 & b_1& &  & -\overline{b}_n\\
-\overline{b}_1& 0 & b_2 &  & \\
& -\overline{b}_2 & 0 & \ddots & \\
 &  & \ddots &\ddots&  b_{n-1}\\
b_n& & & -\overline{b}_{n-1} &0
\end{pmatrix}.
\]
The Toda flow (the flow of the periodic Toda lattice) is the flow
\begin{equation}\label{eqTodaFlow}
\dot{L}=[L,P]=LP-PL.
\end{equation}
The solution $L(t)$ to \eqref{eqTodaFlow} remains similar to the initial matrix $L(0)$ at
all times $t\in\Ro$, so the Toda flow preserves the spectrum. Therefore
the flow acts on the isospectral space~$\Xn$.
\end{con}

\begin{rem}
The Toda flow commutes with the torus action. Indeed, the
action of $T$ on $L$ is given by $DLD^{-1}$, for diagonal
Hermitian matrix $D$. We have $P(DLD^{-1})=DP(L)D^{-1}$ and
therefore $[DLD^{-1},P(DLD^{-1})]=D[L,P(L)]D^{-1}$.
\end{rem}

The periodic Toda system is well studied for real symmetric
matrices. We need a more general Hermitian version of periodic
Toda system in order to incorporate torus actions. However, the complex case
is not more complicated than the real one. The equations of the
flow have the coordinate form:
\begin{equation}\label{eqTodaCoords}
\begin{cases}
\dot{a}_i=2(|b_{i-1}|^2-2|b_i|^2),\quad i=1,\ldots,n;\\
\dot{b}_i=b_i(a_i-a_{i+1}),\quad i=1,\ldots,n,
\end{cases}
\end{equation}
where $a_i,b_i$ are assumed cyclically ordered. Since $b_i\in\Co$, each expression in the second line represents
two real equations. We see that the arguments of $b_i\in\Co$
remain constant along the flow. The equations on $a_i,|b_i|$ have
the form
\begin{equation}\label{eqTodaCoordsAbsValues}
\begin{cases}
\dot{a}_i=2(|b_{i-1}|^2-2|b_i|^2),\quad i=1,\ldots,n;\\
\frac{d}{dt}|b_i|=|b_i|(a_i-a_{i+1}),\quad i=1,\ldots,n,
\end{cases}
\end{equation}
which \edtt{coincides} with the real form of the periodic Toda flow.

\begin{con}\label{conDegenerationPoints}
It is a simple exercise that the quantity $B=\prod_1^n b_i$ is
preserved along the flow. In what follows we consider the
exceptional subspace\edtt{
\[
\Xn^{0}=\{L\in \Xn\mid B=0 \}.
\]
}
This subspace can be represented as the union
$\Xn^{0}=\bigcup_1^nY_i$, where $Y_i\subset \Xn$ is the subset of
matrices having $b_i=0$ for a particular index $i\in[n]$. The set $Y_n$ is
just the set of isospectral tridiagonal Hermitian matrices, which
is known to be a smooth manifold whose smooth type is independent
of a simple spectrum $\lambda$ \cite{Tomei}. Moreover, it is known
that $Y_n$ is a quasitoric manifold over a permutohedron
\cite{BFR,DJ} (the reader is advised to consult \cite{BPnew}
concerning the terminology of quasitoric manifolds). Each of $Y_i$
for $i\neq n$ is diffeomorphic to $Y_n$. This follows from the
fact that the matrix with $b_i=0$ can be transformed to
tridiagonal Hermitian matrix by a cyclic permutation of rows and
columns.

Therefore $\Xn^{0}$ is the union of $n$ submanifolds of dimension
$2n-2$, however, these submanifolds intersect nontrivially. In the
intersection of $Y_i$ and $Y_j$ there lies the submanifold of
matrices with $b_i=b_j=0$, which is a torus invariant codimension
2 submanifold of both $Y_i$ and $Y_j$. The combinatorial structure
of these intersections will be described in detail in Section
\ref{secTilings}.

The Toda flow degenerates to a Toda flow of non-periodic Toda lattice on
the exceptional set $\Xn^{0}$. Each submanifold $Y_i$ is preserved
by the flow. The Toda flow on~$Y_i$ is a gradient flow
(see e.g. \cite{TomeiON} or \cite{ChShS}), which means that
asymptotically each trajectory on~$Y_i$ tends to an equilibrium
point. The equilibrium points are the diagonal matrices
\edtt{
\[
L_\sigma=\diag(\lambda_{\sigma(1)},\ldots,\lambda_{\sigma(n)}),
\quad \sigma\in S_n.
\]
}
A direct check shows that the subspace $\Xn$ is a smooth manifold
in a neighborhood of each equilibrium point $L_\sigma$
\cite{Tomei}. The asymptotical properties of the flow on the
exceptional set imply that $\Xn$ is a smooth manifold in a
neighborhood of $\Xn^{0}$. It will be shown in Section
\ref{secOrbitSpace} that $\Xn$ is not always smooth in points with
large values of $B$.
\end{con}

\begin{rem}\label{remSard}
For generic spectrum $\lambda$ the whole space $\Xn$ is a smooth manifold. This
easily follows from Sard's theorem applied to the map sending the periodic
tridiagonal matrix $L$ to the tuple $(\tr L,\tr L^2,\ldots,\tr L^n)$\edtt{.}
\end{rem}

\section{The orbit space of the torus action}\label{secOrbitSpace}

The action of $T=T^{n-1}$ on $\Xn$ has $n!$ fixed points
$L_\sigma$, $\sigma\in S_n$ which coincide with the equilibria of
the Toda flow.

\begin{prop}\label{propSmoothOverZero}
The orbit space $\Qn=\Xn/T$ is a topological manifold in a
neighborhood of $\Xn^{0}/T$. The space $\Qn$ is a topological
manifold for generic $\lambda$.
\end{prop}

\begin{proof}
Note that $\dim \Xn=2n$ and $\dim T=n-1$. Consider any fixed point
$L_\sigma$. The tangent representation of the action at a point
$L_\sigma$ has the weight decomposition
\[
T_{L_\sigma}\Xn=V(\alpha_{1,\sigma})\oplus\cdots\oplus
V(\alpha_{n,\sigma}),\qquad \alpha_{i,\sigma}\in \Hom(T^n,S^1),
\]
where $V(\alpha)$ is the 1-dimensional complex representation
\[
tz=\alpha(t)\cdot z.
\]
In terms of the noneffective action of $n$-dimensional torus $T^n$
we have
\[
\alpha_{i,\sigma}=\epsilon_i-\epsilon_{i+1},\quad\mbox{ for any
}\sigma\in S_n,
\]
where $\{\epsilon_1=\epsilon_{n+1},\epsilon_2,\ldots,\epsilon_n\}$
is the standard basis of $\Hom(T^n,S^1)\cong\Zo^n$, as follows
from the explicit expression \eqref{eqTorusActionCoords} for the
action.

The following fact was proved in \cite{AyLoc}. Suppose a torus $T$
of dimension $n-1$ acts effectively on a smooth manifold $X$ of
dimension $2n$, and assume that each connected component of each
equivariant skeleton $X_j$ contains a fixed point. Assume,
moreover, that the action has finitely many fixed points, and, at
each fixed point, any $n-1$ of $n$ weights
$\alpha_1,\ldots,\alpha_n\in\Zo^{n-1}$ of the tangent
representation are linearly independent. Then $X/T$ is a closed
topological $(n+1)$-manifold. Applying this result to $\Xn$ in a
neighborhood of $\Xn^0$, we get the first part of the proposition.

The second part follows easily from Remark \ref{remSard}, since the action of $T$ outside
$\Xn^0$ is free. Therefore, whenever $\Xn$ is a smooth manifold,
the orbit space $\Xn/T$ is smooth outside $\Xn^0/T$, thus it is a
topological manifold.
\end{proof}

To describe the topology of $\Qn$ and $\Xn$, we formulate \edt{a} result
of an independent interest. Let $p\colon\Xn\to\Co$ be the map
which associates the number $B=\prod_{i=1}^n b_i$ to a periodic
tridiagonal matrix $L(\underline{a},\underline{b})$. Since the
$T$-action preserves $B$, there is an induced continuous map
$\po\colon\Qn\to\Co$.

The aim of the following constructions is to describe the image of
$\po$ and all its preimages. The description is given in Theorem
\ref{thmImage} below.

\begin{con}
Let a simple spectrum $(\lambda_1<\ldots<\lambda_n)$ be given.
Consider the characteristic polynomial
$F(x)=\prod_{i=1}^n(x-\lambda_i)$. Since the polynomial has $n$
real roots, we have the sequence of real numbers
\edt{
\[
x_1<x_2<\cdots<x_{n-2}<x_{n-1},
\]
}
where $x_{n-1},x_{n-3},x_{n-5},\ldots$ are the local minima, and
$x_{n-2},x_{n-4},\ldots$ are the local maxima of $F$. Let
\begin{equation}\label{eqMMdefin}
M=\min_{i\mbox{ is even}}F(x_{n-i}),\qquad m=\min_{i\mbox{ is
odd}} -F(x_{n-i}).
\end{equation}
We obviously have $m,M>0$.

\begin{figure}[h]
\begin{center}
\includegraphics[scale=0.5]{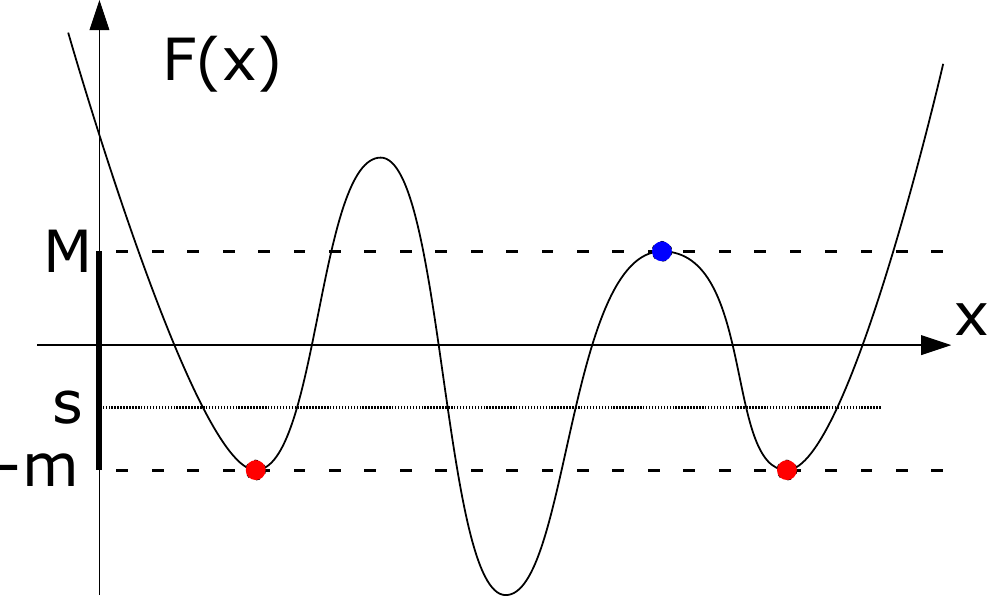}
\end{center}
\caption{The values $M$ and $-m$ on the plot of a characteristic
polynomial}\label{figCharPolyPlot}
\end{figure}

\begin{rem}\label{remMandMmeaning}
The interval $[-m,M]$ represents the set of all $s\in\Ro$ such
that \edt{all roots of} the polynomial $F(x)-s$ \edt{are real}, see
Fig.\ref{figCharPolyPlot}.
\end{rem}

Let $n_+$ be the number of local maxima at which $M$ is achieved
and, similarly, $n_-$ be the number of local minima at which $-m$
is achieved. For generic $\lambda$ there holds $n_+=n_-=1$.
Fig.\ref{figCharPolyPlot} shows the case $n_+=1$, $n_-=2$.
\end{con}

\begin{thm}\label{thmImage}
The image of $p\colon\Xn\to\Co$ is the set
\begin{equation}\label{eqImage}
\B=\left\{z\in\Co\left| |z|\leqslant
\frac12\min\left(\frac{m}{1+\cos\Arg z},\frac{M}{1-\cos\Arg
z}\right)\right.\right\}.
\end{equation}
The preimages of the map $\po\colon \Qn\to\B$ are as follows. If
$z\in \B^{\circ}$, then $\po^{-1}(z)$ is homeomorphic to a compact
torus $\T^{n-1}$. If $z\in\dd\B$ and minimum in \eqref{eqImage} is
achieved at $\frac{M}{1-\cos\Arg z}$, then $\po^{-1}(z)$ is a
torus of dimension $n-1-n_+$. If $z\in\dd\B$ and minimum in
\eqref{eqImage} is achieved at $\frac{m}{1+\cos\Arg z}$, then
$\po^{-1}(z)$ is a torus of dimension $n-1-n_-$. If $z\in\dd\B$
and $\frac{m}{1+\cos\Arg z}=\frac{M}{1-\cos\Arg z}$, then
$\po^{-1}(z)$ is a torus of dimension $n-1-n_+-n_-$.
\end{thm}

The convex set $\B$ is shown on Fig.\ref{figBset}: it is bounded
by arcs of two confocal parabolas. The set $\B$ is a 2-dimensional
manifold with corners: we denote by $F_+$ and $F_-$ its left and
right sides respectively, and $F_+\cap F_-=\{z_{top},z_{bot}\}$.
Note that the minimum is achieved at $\frac{M}{1-\cos\Arg z}$
whenever $z$ lies on the left side of the figure, which explains the notation.

\begin{rem}
It will be convenient to distinguish between the torus, which acts
on spaces\edt{,} and the geometrical tori arising in Theorem
\ref{thmImage}. Hence toric groups are denoted by the symbol~$T$,
and tori appearing in geometrical considerations are denoted by
the symbol~$\T$.
\end{rem}

\begin{figure}[h]
\begin{center}

\begin{tikzpicture}[scale=0.6]
\filldraw[color=black!15]
plot[domain=-2*sqrt(3):2*sqrt(3),smooth,variable=\t]
({-(9-\t*\t)/6},{\t})--plot[domain=2*sqrt(3):-2*sqrt(3),smooth,variable=\t]
({-(\t*\t-16)/8},{\t});

\draw[->] (-3,0) -- (3,0); 

\draw[->] (0,-5) -- (0,5); 

\draw[domain=-2:2,smooth,variable=\t,blue,ultra thick] plot
({-3*cos(\t r)/(1+cos(\t r))},{3*sin(\t r)/(1+cos(\t r))});

\draw[domain={-1.7+pi}:{1.7+pi},smooth,variable=\t,red,ultra
thick] plot ({-4*cos(\t r)/(1-cos(\t r))},{4*sin(\t r)/(1-cos(\t
r))});

\node[color=blue] at (-2,1) {$F_+$}; \node[color=red] at (2.5,1)
{$F_-$};

\filldraw (0.5,{2*sqrt(3)}) circle(3pt) (0.5,{-2*sqrt(3)})
circle(3pt);

\node at (0.5,0.5) {$\B$}; \node at (0.7,4.3) {$z_{top}$}; \node
at (0.7,-4.3) {$z_{bot}$};

\end{tikzpicture}
\end{center}
\caption{The set $\B$}\label{figBset}
\end{figure}

\begin{cor}\label{corOrbitSpaceGeneral}
With parameters $n_+$ and $n_-$ as above, the orbit space $\Qn$ is
homeomorphic to $\Sigma(\T^{n_-}\ast \T^{n_+})\times
\T^{n-1-n_--n_+}$.
\end{cor}

\begin{proof}[Proof of the corollary]
The space $\Qn$ is foliated over the contractible space $\B$ by
tori. Hence
\[
\Qn\cong \B\times \T^{n-1}/\sim,
\]
where certain $n_+$-dimensional subtorus $\T_+$ is collapsed over
$F_+$ and another $n_-$-dimensional subtorus $\T_-$ is collapsed
over $F_-$ (the nature of these tori and their independence is
clarified in Section \ref{secSpectralCurve}). We have
$\T=\T_+\times \T_-\times \T^{n-1-n_--n_+}$. The sub\edt{torus}
$\T^{n-1-n_--n_+}$ separates as a direct factor of $\Qn$. The
remaining factor is the suspension space, with the suspension
points being the preimages of the points $z_{top}$ and $z_{bot}$.
This suspension is taken over the space \edt{$\po^{-1}(\B\cap\Ro)$}
which is homeomorphic to the join of $\T_+$ and $\T_-$.
\end{proof}

\begin{cor}\label{corOrbitSpace}
For generic spectrum $\lambda$ there is a homeomorphism $\Qn\cong S^4\times
\T^{n-3}$.
\end{cor}

\begin{proof}
In generic case we have $n_+=n_-=1$. Therefore
$\Sigma(\T^1\ast\T^1)\cong\Sigma S^3~\cong~S^4$.
\end{proof}

\begin{cor}[Theorem of Buchstaber--Terzic \cite{BTober,BT2}]\label{corOrbitFlags}
Consider the effective action of $T=T^3/\Delta(T^1)$ on the
manifold $\Fl_3$ of complete complex flags in $\Co^3$. The orbit
space $\Fl_3/T$ is homeomorphic to $S^4$.
\end{cor}

\begin{proof}
Note that $X_{3,\lambda}$ is just the set of all Hermitian
matrices with the given spectrum. This manifold is diffeomorphic
to the flag manifold $\Fl_3$. \edt{Notice that, if $n=3$ the characteristic polynomial is cubic
so it has exactly one local maximum and one local minimum: $n_+=n_-=1$.
Hence the statement} is the particular case
of Corollary \ref{corOrbitSpace} with $n=3$.
\end{proof}

\begin{thm}\label{thmNonSmooth}
If $\lambda$ is a simple spectrum such that either $n_+>1$ or $n_->1$,
then $\Xn$ is not a homology manifold. In
particular, this space is not a smooth manifold.
\end{thm}

\begin{proof}
The space $\Qn\cong \Sigma(T^{n_-}\ast
T^{n_+})\times T^{n-1-n_--n_+}$ is not a homology manifold unless
$n_+=n_-=1$. \edt{Indeed, assume $n_->1$ (the case $n_+>1$ is completely similar). Let $q\in\Qn$ be a point such
that $\po(q)$ lies in the interior of $F_+$. There exists a neighborhood $U_q\subset\Qn$ of
$q$ homeomorphic to $\Cone T^{n_-}\times \Ro^{n-n_-}$, according to the
structure of the join. We have
\[
H_i(U_q,U_q\setminus\{q\}) \cong \Hr_{i-1}(T^{n_-}\ast S^{n-n_--1})\cong \Hr_{i-1-n+n_-}(T^{n_-})\neq 0
\]
for \edtt{$n-n_-+1<i<n+1=\dim\Qn$}, that is at least for one value of $i$. This argument shows that $\Qn$
fails to be a homology manifold at $q$.}

\edt{
The torus action on $\Xn$ is free over $\dd\B$. Hence, for any
point $x\in\Xn$ lying in the orbit $q$, its neighborhood $U_x\ni
x$ is homeomorphic to $U_q\times\Ro^{n-1}$. Therefore
\[
H_i(U_x,U_x\setminus\{x\})\cong H_{i-n+1}(U_q,U_q\setminus\{q\}) \neq 0
\]
for some $i<2n$, so far $\Xn$ fails to be a homology manifold at $x$.}
\end{proof}

\begin{rem}
Van Moerbeke \cite{VanM} proves the real analogue of Theorem
\ref{thmImage}. In the real case, there is a family of tori,
parametrized by real numbers from the interval
$[-M/4,0)\subset\B\cap\Ro$. The dimension of all tori is $n-1$,
except for the torus over the endpoint $-M/4$: its dimension
decreases by $n_+$. Van Moerbeke calls the union of such family ``an
open $n$-dimensional torus''. This naming seems misleading, since
this union is not even a manifold for $n_+>1$, which is proved
similarly to Theorem \ref{thmNonSmooth}.
\end{rem}

\begin{rem}\label{remDegenerateCheb}
The most degenerate case appears in the situation when the values at
all local minima of the characteristic polynomial $F(x)=\prod_{i=1}^{n}(x-\lambda_i)$
coincide, and the values at all local maxima coincide.
For example, this holds for the Chebyshev polynomials $T_n(x)$ (which
are defined on the interval $[-1,1]$ by $T_n(x)=\cos(n\arccos x)$). It can be
seen, that a polynomial gives the maximal possible degeneration if and only if
it coincides with Chebyshev polynomial up to affine transformation of the image and the domain:
\[
F(x)=\gamma T_n(\alpha x+\beta)+\delta,
\]
with the natural requirement that $F(x)$ \edt{is monic} and has $n$ distinct real roots.
\end{rem}

\section{\Sch equation and the spectral
curve}\label{secSpectralCurve}

In this section we prove the first part of Theorem \ref{thmImage}.
It will be assumed that $B=\prod_1^nb_i\neq 0$, i.e. $L\notin
\Xn^0$. The action of $T$ is free on such matrices. We may
identify $\Qn=\Xn/T$ with the set of isospectral Hermitian
matrices of the form
\begin{equation}\label{eqRotatedMatrix}
L(w)=\begin{pmatrix}
a_1 & b_1& 0&\cdots& w^{-1}b_n\\
b_1& a_2 & b_2 & & \vdots\\
0 & b_2 & a_3 & \ddots &\\
\vdots&&\ddots& \ddots& b_{n-1}\\
wb_n& \cdots&& b_{n-1} &a_n
\end{pmatrix}
\end{equation}
where $b_1,\ldots,b_n$ are positive real numbers, and $w\in\Co$,
$|w|=1$. Indeed, the arguments of any $n-1$ off-diagonal terms of
a periodic tridiagonal matrix $L(\underline{a},\underline{b})$ can
be rotated to zero by the torus action \eqref{eqTorusActionCoords}. We continue denoting
$\prod_{i=1}^n b_i$ by the letter $B$, although in the new
notation $B$ is a positive real number.

\begin{prop}\label{propImageDescription}
For a matrix $L(w)$ with a simple spectrum
$\lambda_1<\cdots<\lambda_n$ there holds
\[
B \leqslant \frac12\min\left(\frac{M}{1-\cos\Arg
w},\frac{m}{1+\cos\Arg w}\right),
\]
where $M$ and $m$ are defined by \eqref{eqMMdefin}.
\end{prop}

\begin{proof}
Matrices of the form $L(w)$ can be studied using algebro-geometric
method in mathematical physics (we refer to \cite{Krich} for a
brief exposition of this subject in relation to periodic Toda
flow). Let $\l$ be the space of infinite to both sides sequences
$\{\psi_k\}$:
\edtt{
\[
\psi_k\in\Co,\quad k\in\Zo.
\]
}
Consider the \emph{periodic discrete \Sch} operator given by
\[
H\colon \l\to \l,\qquad
H(\psi)_k=b_{k-1}\psi_{k-1}+a_k\psi_k+b_k\psi_{k+1},
\]
where we assume $a_{k+n}=a_k$ and $b_{k+n}=b_k$. The eigenfunction
$\psi\in\l$ of the \Sch operator with eigenvalue $x$ satisfies the
equation
\edtt{
\begin{equation}\label{eqSchEq}
H(\psi)=x\psi.
\end{equation}
}
Since $b_i\neq 0$, every eigenfunction is determined by its
initial values $(\psi_0,\psi_1)\in\Co^2$. We can define the
\emph{monodromy operator} along the period:
\begin{equation}\label{eqMonodromy}
M(x)\colon \Co^2\to\Co^2, \qquad M(x)\colon (\psi_0,\psi_1)\mapsto
(\psi_n,\psi_{n+1}).
\end{equation}
Note that the matrix $L(w)$ has eigenvalue $x$ if and only if
there exists a solution $\psi$ to \eqref{eqSchEq} such that
\edtt{
\[
\psi_{k+n}=w\psi_k.
\]
}
Such functions are called \emph{Bloch solutions}. We see that
whenever there exists a nonzero Bloch solution with parameter $w$,
the number $w$ is the eigenvalue of the monodromy operator $M(x)$,
so we get a relation\edt{
\begin{equation}\label{eqSpecCurve}
\det(wI-M(x))=0.
\end{equation}
}
This equation defines a so called \emph{spectral curve} of the
periodic \Sch equation in the space of parameters $(w,x)\in\Co^2$.
\edt{We have $\det M(x)=1$. Indeed, the operator $M_i\colon
(\psi_{i-1},\psi_i)\mapsto (\psi_i,\psi_{i+1})$ has determinant
$\frac{b_{i-1}}{b_i}$, therefore
\[
M(x)=M_nM_{n-1}\cdots M_1=\frac{b_{n-1}}{b_n}\frac{b_{n-2}}{b_{n-1}}\cdots\frac{b_0}{b_1}=1.
\]}
Hence, the equation \eqref{eqSpecCurve} of the spectral curve can
be rewritten in the form
\begin{equation}\label{eqSpecCurve2}
w^2-\tr M(x)w+1=0.
\end{equation}
\edt{We have $\tr M(x)=\frac{1}{B}P(x)$, where
$B=\prod_1^n b_i$ as before, and $P(x)$ is a monic polynomial in~$x$.
Indeed, decomposing $M(x)$ as the product of operators $M_i$
along the period, $i=1,\ldots,n$, and counting the terms of highest
degree of $x$, we see that the leading coefficient of $\tr M(x)$ is $1/B$}.
Dividing \eqref{eqSpecCurve2} by $w$ and denoting
$t=\re w=\frac12(w+w^{-1})$, we get $2t=\frac{1}{B}P(x)$.

The polynomial $P(x)-2Bt$ is monic and has the given sequence
$\lambda_1,\ldots,\lambda_n$ as its roots, therefore
\[
P(x)-2Bt=\prod\nolimits_{i=1}^n(x-\lambda_i)=F(x),
\]
\edtt{
\begin{equation}\label{eqRelOnPolynomials}
P(x)=F(x)+2Bt.
\end{equation}
}
Consider the set
\[
\A=\{s\in\Ro\mid P(x)=2Bs\mbox{ has }n \mbox{ real roots}\}.
\]
Recalling the definition of $m$ and $M$ and remark
\ref{remMandMmeaning} as well as relation
\eqref{eqRelOnPolynomials}, we see that $\A$ is the closed
interval $[-\frac{m}{2B}+t,\frac{M}{2B}+t]$.

Note that the polynomial $P(x)-2Bs$ is the characteristic
polynomial of the matrix $L(w_s)$, where $|w_s|=1$, $\re w_s=s$. \edt{The matrix $L(w_s)$ is Hermitian,
therefore all its eigenvalues are real. Hence, for any $s\in[-1,1]$, all roots of the equation $P(x)=2Bs$ are real.} Therefore,
\[
[-1;1]\subseteq \left[-\frac{m}{2B}+t,\frac{M}{2B}+t\right],
\]
from which we deduce $B\leqslant \frac12\min(\frac{M}{1-t},
\frac{m}{1+t})$. Remembering $t=\re w=\cos\Arg w$, we get the
required inequality.
\end{proof}

\begin{prop}\label{propPreimageOfNonzero}
For $z\in\B$, $z\neq 0$, the preimage $\po^{-1}(z)$ is
homeomorphic to a torus. The dimension of a torus is $n-1$ if $z$
lies in the interior of $\B$, $n-1-n_+$ if $z$ lies in the
relative interior of $F_+$, $n-1-n_-$ if $z$ lies in the relative
interior of $F_-$, and $n-1-n_+-n_-$ if $z$ is either $z_{top}$ or
$z_{bot}$.
\end{prop}

\begin{proof}
In short, this follows from the fact that the periodic Toda lattice is an
integrable dynamical system and its energy levels are the compact
submanifolds. Liouville--Arnold theorem then implies that these
preimages $\po^{-1}(z)$ are tori. To specify the dimensions we
give more details on the theory, related to periodic tridiagonal
matrices.

As before, consider $P(x)=B\tr M(x)$, the monic polynomial in $x$
with coefficients depending on $a_i,b_i\in\Ro$. As follows from
the considerations above, the eigenvalues of matrices $L(1)$ and
$L(-1)$ are the roots of the polynomials $P(x)-2B$ and $P(x)+2B$
respectively. Let
\[
x_1<x_2\leqslant x_3<\cdots<x_{2n-2}\leqslant x_{2n-1}<x_{2n}
\]
be the union of all these roots, so that
$x_{2n},x_{2n-3},x_{2n-4},x_{2n-7},x_{2n-8},\ldots$ are the roots
of $P(x)-2B$ and $x_{2n-1},x_{2n-2},x_{2n-5},x_{2n-6},\ldots$ are
the roots of $P(x)+2B$. The intervals
\[
I_1=[x_{2},x_{3}],\quad I_2=[x_{4},x_{5}],\quad\ldots\quad,
I_{n-1}=[x_{2n-2},x_{2n-1}]
\]
are called the \emph{forbidden zones}. We will call
$I_{n-1},I_{n-3},\ldots$ lower forbidden zones, and
$I_{n-2},I_{n-4},\ldots$ upper forbidden zones as motivated by
Fig.\ref{figForbiddenZones}.

\begin{figure}[h]
\begin{center}
\includegraphics[scale=0.6]{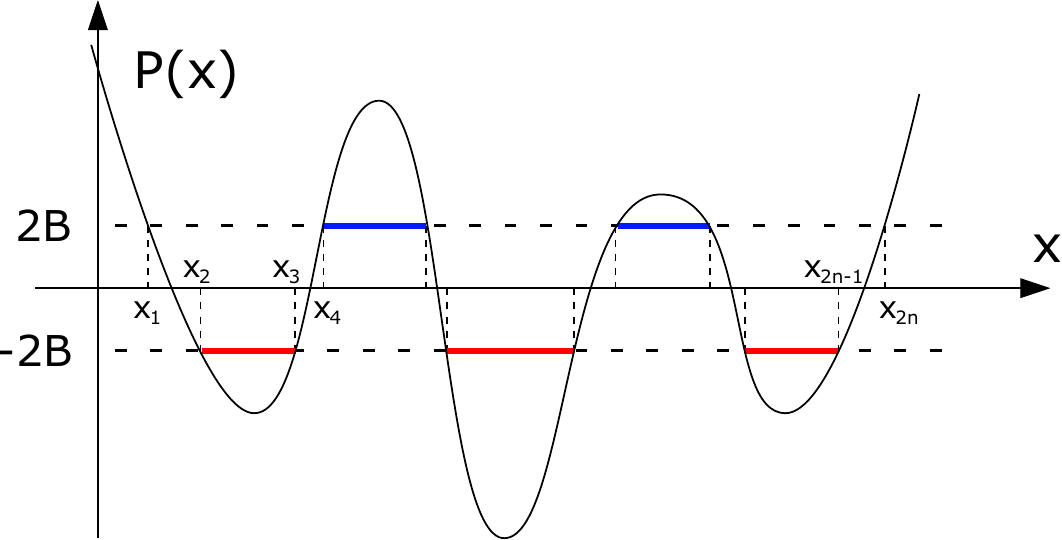}
\end{center}
\caption{Upper and lower forbidden zones}\label{figForbiddenZones}
\end{figure}

Consider the Riemannian surface $\Theta_g$ of the multivalued
function
\edtt{
\[
g(x)=\sqrt{\prod\nolimits_{i=1}^{2n}(x-x_i)}.
\]
}
Over each forbidden interval $I_k$, $k=1,\ldots,n-1$, there lies a
circle $S_k$ on $\Theta_g$. If an interval $I_k$ degenerates to a
point (i.e. $x_{2k}=x_{2k+1}$), the circle $S_k$ also collapses to
a point. Van Moerbeke \cite{VanM} proved

\begin{prop}\label{propMoerbeke}
Real periodic tridiagonal symmetric matrices
$L(\underline{a},\underline{b})$ with the given spectrum $\lambda$, given $B=\prod_1^nb_i$,
and $b_i>0$, are in one-to-one correspondence with
$(n-1)$-tuples $(\mu_1,\ldots,\mu_{n-1})$, where $\mu_k\in S_k$.
\end{prop}

Therefore, for real $z$, the preimage $\po^{-1}(z)$ is
diffeomorphic to a torus $\T=\prod_1^{n-1} S_i$. The dimension of
this torus equals $n-1$ in general, however, when some forbidden
intervals are collapsed, the dimension reduces by the number of
collapsed intervals. The upper forbidden intervals collapse if and
only if the value $2B$ reaches $M$. The number of collapses among
upper intervals equals $n_+$. Similarly, the lower intervals
collapse if $2B$ reaches $m$, and the number of collapses among
lower intervals is $n_-$.

Now let $L(w)$ be an arbitrary matrix with $w\in\Co$, $|w|=1$ and
the given spectrum $\lambda$. Let $\widetilde{\lambda}$ be the set
of roots of the polynomial $\prod(x-\lambda_i)+2B\re w$. It was
mentioned in the proof of Proposition \ref{propImageDescription}
that the matrix $L(w)$ has spectrum $\lambda$ if and only if
$L(1)$ has spectrum $\widetilde{\lambda}$. Thus Proposition
\ref{propMoerbeke} implies the required statement for all
matrices.
\end{proof}

\section{Permutohedral tilings}\label{secTilings}

In this section we study the degenerate locus of the periodic Toda
lattice. Recall that $\Xn^{0}=p^{-1}(0)\subset \Xn$ is the set of all
isospectral matrices with $B=\prod_1^nb_i=0$, and
$\Qn^{0}=\Xn^{0}/T=\po^{-1}(0)$.


We recall some standard facts from combinatorial geometry. Let
$\epsilon_1,\ldots,\epsilon_n$ be the standard basis of
$\Zo^n\cong\Hom(T^n,S^1)$. We assume that $\Zo^n\subset \Ro^n$ and
there is a fixed inner product on $\Ro^n$ such that
$\epsilon_1,\ldots,\epsilon_n$ are orthonormal.

Consider the sublattice $N\subset \Zo^n$ of rank $n-1$ given by \edtt{
\[
N=\left\{\left.\sum\nolimits_{i=1}^na_i\epsilon_i\right|
a_i\in\Zo, \sum a_i=0,\mbox{ and } a_i-a_j\equiv 0\mod n\right\},
\]
}
and let $N_\Ro=N\otimes_\Zo\Ro$ be its real span. Consider the vectors
$\alpha_1,\ldots,\alpha_n$:
\edt{
\[
\alpha_i=(n-1)\epsilon_i-\sum_{j\neq i}\epsilon_j,\qquad i=1,\ldots,n.
\]
}
We see that
\begin{equation}\label{eqAlphaNotation}
\sum_{i=1}^n\alpha_i=0,
\end{equation}
and any $n-1$ of $\alpha_1,\ldots,\alpha_n$ generate the lattice $N$. One
can think about $\alpha_i$'s as the \edt{inward normal} vectors to
the facets of a regular simplex in $\Ro^{n-1}$.

For any subset $S\subset[n]=\{1,\ldots,n\}$ such that
$S\neq\varnothing, [n]$, consider the vector
\begin{equation}\label{eqAlphaSconvention}
\alpha_S=\sum\nolimits_{i\in S}\alpha_i.
\end{equation}
Let $\Pa_{n-1}$ be the Voronoi cell decomposition of $N_\Ro\cong \Ro^{n-1}$
generated by the lattice $N$. In other words, for any $\alpha\in
N$ we consider the Voronoi cell
\[
P_\alpha=\{x\in N_\Ro\mid \dist(x,\alpha)\leqslant \dist(x,\beta)
\mbox{ for any }\beta\in N,\beta\neq\alpha\},
\]
where $\dist$ is the distance determined by the inner product on
$N_\Ro\subset \Ro^n$. Each $P_\alpha$ is a convex
$(n-1)$-dimensional polytope and all these polytopes are the
parallel copies of each other, $P_\alpha=P_0+\alpha$.

\begin{con}
It can be shown \edt{(see details in \cite[Sec.3.8, Thm.7]{ConSloane})} that $P_0$ is the $(n-1)$-dimensional
permutohedron $\Pe^{n-1}$ determined by the inequalities
\[
\Pe^{n-1}=P_0=\left\{x\in N_\Ro\mid \langle \alpha_S,x\rangle\leqslant
\frac12\langle\alpha_S,\alpha_S\rangle, S\in 2^{[n]},
S\neq\varnothing,[n]\right\},
\]
\edt{where $\langle\cdot,\cdot\rangle$ is the inner product on $N$}.
We recall the basic facts about the combinatorics of a
permutohedron. The polytope $\Pe^{n-1}$ is simple, which means
that every codimension $k$ face is contained in exactly $k$
facets. For a proper subset $S\subset[n]$ let $F_S$ denote the
facet of $\Pe^{n-1}$ determined by the support hyperplane
$\langle\alpha_S,x\rangle=
\frac12\langle\alpha_S,\alpha_S\rangle$. Note that $\Pe^{n-1}$ is
centrally symmetric: the facets $F_S$ and $F_{\bar{S}}$ are
opposite to each other whenever $\bar{S}=[n]\setminus S$.

Facets $F_{S_1},\ldots,F_{S_k}$ have nonempty intersection in
$\Pe^{n-1}$ if and only if the subsets $\{S_1,\ldots,S_k\}$ form a
chain in the Boolean lattice $2^{[n]}$. If
$\sigma=(S_1\subset\cdots\subset S_k)$ is such a chain, we denote
by $F_\sigma$ the face $F_{S_1}\cap\cdots\cap F_{S_k}$ of the
permutohedron. Each face of $\Pe^{n-1}$ is known to be a product
of permutohedra of smaller dimensions.

We denote by $F_S(P_\alpha)$ (resp. $F_\sigma(P_\alpha)$) the
corresponding facets (resp. faces) of the the Voronoi cell
$P_\alpha$ to distinguish different copies of a permutohedron in
the Voronoi diagram. It can be seen that
\begin{equation}\label{eqNeighbors}
F_S(P_\alpha)=F_{\bar{S}}(P_{\alpha+\alpha_S}).
\end{equation}
A facet of each cell is \edt{adjacent} to an opposite facet of a neighboring
cell.
\end{con}

We formulate a general construction to precede a particular case
needed in the proof of Theorem \ref{thmImage}.

\begin{con}\label{conGeneralVoronoiQuotient}
Let $\wh{N}\subseteq N$ be a sublattice of finite index, i.e.
$q=|N/\wh{N}|<\infty$. Consider the quotient $N_\Ro/\wh{N}$. Since
$\wh{N}$ is a cocompact lattice, this quotient is a torus
$\T^{n-1}$. The action of $\wh{N}$ by parallel shifts preserves
the Voronoi diagram, therefore we have a cell subdivision of the
torus $\T^{n-1}\cong N_\Ro/\wh{N}$. There are $q$ maximal cells in
this subdivision, each is a parallel copy of a permutohedron.
\end{con}

\begin{ex}
A natural example is $\wh{N}=N$. In this case the torus
is given by identifying the opposite facets of a single
permutohedron. The cell structure on a torus given by this
identification is known: the corresponding partially ordered set
was introduced and studied by Panina
\cite{Panina} under the name of \emph{cyclopermutohedron}. This poset
has a natural combinatorial description.
\end{ex}

For the considerations of this paper we need another sublattice.

\begin{con}\label{conWonderfulSublattice}
Let $N'\subset N$ be the sublattice generated by the vectors
\[
\beta_k=\alpha_k-\alpha_{k+1},\qquad k=1,\ldots,n-1.
\]
Note that
$\beta_{n-1}=\alpha_{n-1}-\alpha_n=2\alpha_{n-1}+\alpha_1+\cdots+\alpha_{n-2}$
according to \eqref{eqAlphaNotation}. It can be shown that
$N/N'$ is the cyclic group of order $n$. Indeed, in the quotient
group $N/N'$ we have the identities
\begin{equation}\label{eqModLattice}
[\alpha_1]=\cdots=[\alpha_n],\qquad
n[\alpha_1]=[\alpha_1]+\cdots+[\alpha_{n-2}]+2[\alpha_{n-1}]=0.
\end{equation}
\end{con}
\begin{defin}\label{definWonderfulDecomposition}
Let $\PT^{n-1}$ be the cell decomposition of a torus $\T^{n-1}$
obtained as a quotient of Voronoi diagram of the space $N_\Ro$ by
the sublattice $N'$. We call $\PT^{n-1}$ the \emph{wonderful cell
decomposition} of a torus.
\end{defin}

The wonderful decomposition $\PT^{n-1}$ has $n$ maximal cells. The
cells $P_\alpha$ and $P_{\alpha+\beta}$ are identified in
$\PT^{n-1}$ whenever $\beta\in N'$. We denote the resulting cell
of $\PT^{n-1}$ by $P_{[\alpha]}$. Relations \eqref{eqModLattice}
imply
\[
[\alpha_S]=|S|[\alpha_1],\qquad [n\alpha_1]=[0].
\]

\begin{lem}\label{lemGluingRulesForZ}
Let $1\leqslant k<m\leqslant n$. In the cell complex $\PT^{n-1}$
we have
\[
F_S(P_{k[\alpha_1]})=F_{\bar{S}}(P_{m[\alpha_1]}),
\]
where $S$ is any subset of $[n]$ of cardinality $m-k$.
\end{lem}

\begin{proof}
Choose any subset $S'$ such that $|S'|=k$ and $S'$ is disjoint from
$S$. According to \eqref{eqNeighbors} we have
\[
F_S(P_{k[\alpha_1]})=F_S(P_{[\alpha_{S'}]})=F_{\bar{S}}(P_{[\alpha_{S'}+\alpha_S]})=
F_{\bar{S}}(P_{[\alpha_{S'\sqcup
S}]})=F_{\bar{S}}(P_{m[\alpha_1]}),
\]
which proves the statement.
\end{proof}

In the following, we denote the maximal cells $P_{k[\alpha_1]}$ by
$\PT_k$. Now we return to the space of tridiagonal matrices.
Recall that $Y_k$ denotes the space of all isospectral matrices
$L(\underline{a},\underline{b})$ with $b_k=0$, for $k=1,\ldots,n$.
Let $Q_k$ denote the orbit space $Y_k/T$. We have
$\Qn^{0}=\bigcup_1^nQ_k$.

\begin{thm}\label{thmTorusForZero}
The space $\Qn^{0}$ can be identified with $\PT^{n-1}$ so that the
subspaces $Q_k$ are identified with $\PT_k$.
\end{thm}

\begin{proof}
The orbit space $Q_0=Q_n$ is identified with the space of all
tridiagonal symmetric real matrices
\[
L=\begin{pmatrix}
a_1 & b_1& 0&\cdots& 0\\
b_1& a_2 & b_2 &  & 0\\
0 & b_2 & a_3 & \ddots & \vdots \\
\vdots&& \ddots& \ddots &b_{n-1}\\
0& 0&\cdots& b_{n-1} &a_n
\end{pmatrix}
\]
with $b_i\geqslant 0$ and the given simple spectrum $\lambda$. It
is known (see \cite{Tomei}) that $Q_0$ is diffeomorphic to a
permutohedron $\Pe^{n-1}$ as a manifold with corners. The facet
$F_S(\Pe^{n-1})$ corresponds to the subset of $Q_0$, which
consists of matrices $L$ such that $b_{|S|}=0$ and the eigenvalues
$\{\lambda_i\mid i\in S\}$ are distributed in the first $(|S|\times
|S|)$-block.

Similar considerations are valid for other spaces $Q_k$: this
can be shown by cyclic permutation of rows and columns of $L$.
Indeed, the set $Q_k$ can be identified with $\Pe^{n-1}$ in such
way that the facet $F_S(\Pe^{n-1})$ consists of all matrices with
the property
\[
b_k=0,\qquad b_{k+|S|}=0,
\]
and the block between $k$-th and $(k+|S|)$ rows and columns has
eigenvalues $\{\lambda_i\mid i\in S\}$.

It can be seen that the faces $F_S(Q_k)$ and $F_{\bar{S}}(Q_{m})$
represent the same set of matrices for $1\leqslant k<m\leqslant n$
and $|S|=m-k$. Therefore, $F_S(Q_k)=F_{\bar{S}}(Q_{m})$ in
$\Qn^0$. These gluing rules for the cells in $\Qn^0$ coincide with
the gluing rules for $\PT_k$ in $\PT^{n-1}$ according to Lemma
\ref{lemGluingRulesForZ}.
\end{proof}

\edt{For convenience introduce the cyclic
notation: $Q_k=Q_{k+n}$, for any $k\in \Zo$.}

\begin{ex}
Right part of Fig.\ref{figHexes} shows the space
$Q_{3,\lambda}^{0}=\PT^2$. This example was described in details
by van Moerbeke \cite{VanM}. The 1-skeleton of $\PT^2$ is shown on the left. As
an abstract graph, it is isomorphic to the complete bipartite
graph $K_{3,3}$. This graph is a GKM-graph of the complete flag
variety $\Fl_3$, see details in \cite{AyLoc}.
\end{ex}

\begin{figure}[h]
\begin{center}
\includegraphics[scale=0.2]{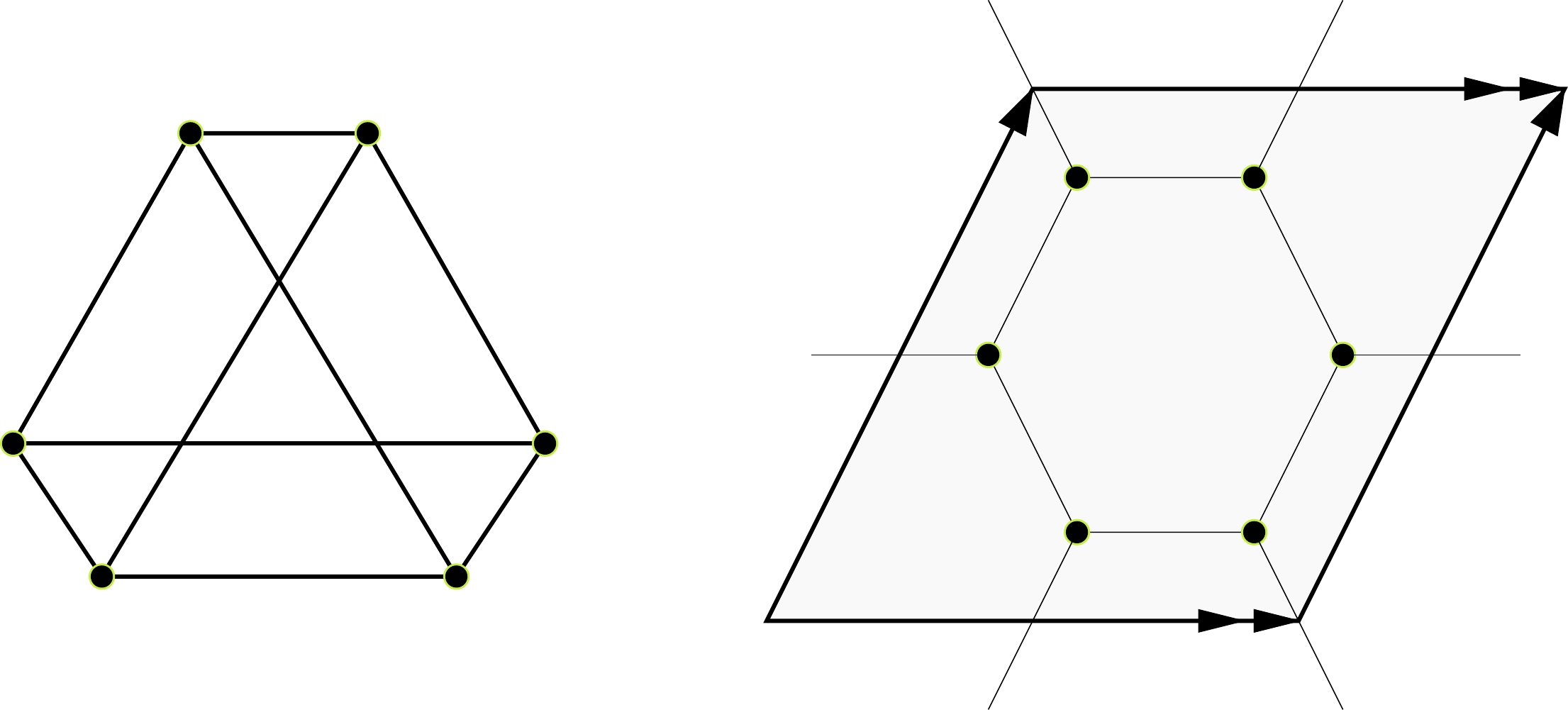}
\end{center}
\caption{The wonderful cell decomposition
$Q_{3,\lambda}^{0}=\PT^2$.}\label{figHexes}
\end{figure}

\begin{rem}
Let us briefly sketch the phase portrait of the Toda flow on the
degenerate set of orbits $\Qn^0$. Let $v\in N_\Ro^*$ be a generic
linear function on $N_\Ro$. Take any face of any permutohedron of
the Voronoi diagram in $N_\Ro$. On each such polytope consider a
flow, which moves all points in the interior of $P$ to the vertex
maximizing the linear function $v$. The flow looks the same
on all Voronoi cells, thus we have an induced flow on the torus
$\PT^{n-1}=N_\Ro/N'$.

This picture describes the Toda flow on $\Qn^0\cong \PT^{n-1}$.
Indeed, Toda flow degenerates to the flow of a non-periodic Toda lattice on each
permutohedron $Q_i$, and its Morse-like behavior is well-known
(see \cite{DNT}). For any block tridiagonal matrix, the Toda flow
``sorts'' the diagonal elements within each block \cite{Tomei}.

The phase portrait for $n=3$ is shown on the left part of
Fig.\ref{figWonderfulFlow}. The oddity of the phase portrait near
equilibria points is explained by the fact that the orbit space
$\Qn$ is not smooth at these points.

Note that for $B\neq 0$, the Toda flow exhibits
Liouville--Arnold behavior. The equilibria points disappear,
however the flow still follows some direction $v$ on a torus, see
Fig.\ref{figWonderfulFlow}, right part.
\end{rem}

\begin{figure}[h]
\begin{center}
\includegraphics[scale=0.3]{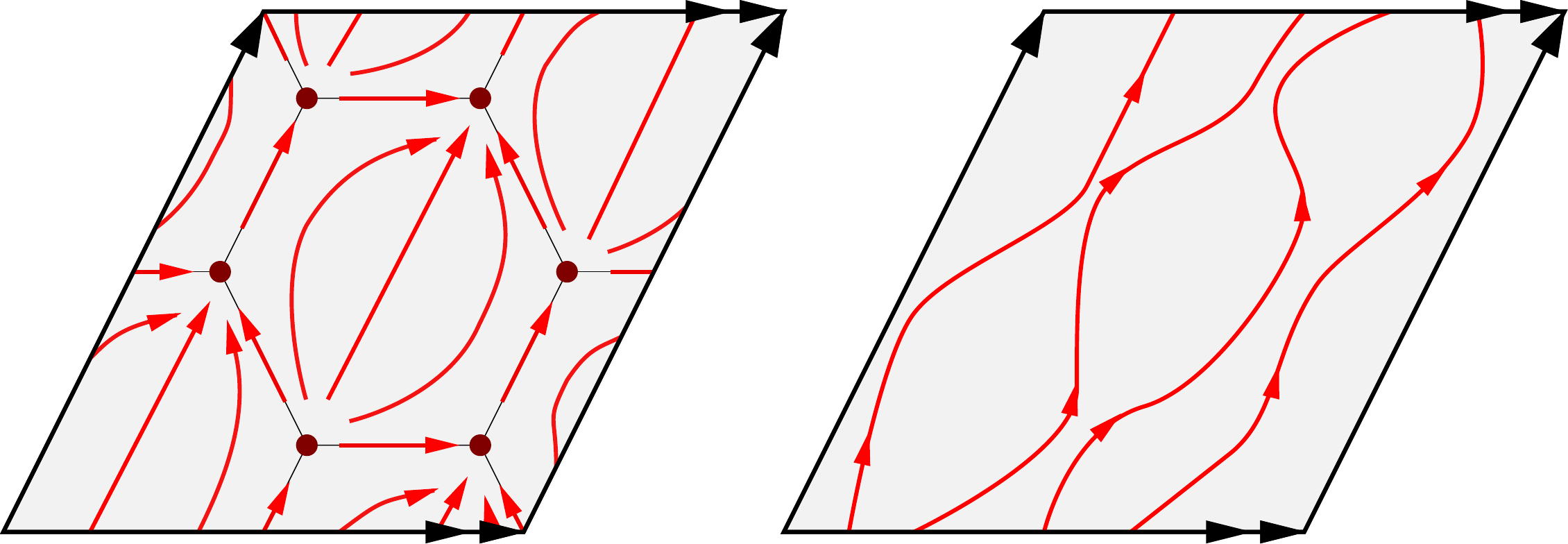}
\end{center}
\caption{The Toda flow on the level set $B=\const$ of the orbit
space $Q_{3,\lambda}^0$. Left part shows the case $B=0$. Right
part shows the case $B=\const\neq 0$}\label{figWonderfulFlow}
\end{figure}

\begin{rem}
Each $k$-dimensional cell of the cell subdivision $\PT^{n-1}$ lies
in exactly $n-k$ different maximal cells. This means there exists
a dual simplicial cell subdivision $\KPT^{n-1}$. In Section
\ref{secNeighborhoodCombinatorics} we recall the definition of a
simplicial poset which is a useful combinatorial notion to study
simplicial cell subdivisions.
%

Note that the simplicial poset $\KPT^{n-1}$ minimizes the number
of vertices among all simplicial cell subdivisions of the torus
$\T^{n-1}$. Indeed, any $(n-1)$-dimensional simplex of such
subdivision has $n$ distinct vertices, therefore a simplicial cell
subdivision of $\T^{n-1}$ should have at least $n$ vertices. This
number is achieved at $\KPT^{n-1}$.
\end{rem}

\begin{rem}
Note that any closed connected $(n-1)$-manifold admits a simplicial cell
subdivision with exactly $n$ vertices. This result was proved in \cite{FGG},
where such subdivisions (or, their equivalent combinatorial representations)
were called \emph{the crystallizations}. Previous remark shows
that $\KPT^{n-1}$ provides an explicit crystallization for the torus $\T^{n-1}$.
\end{rem}

Propositions \ref{propImageDescription},
\ref{propPreimageOfNonzero}, and Theorem \ref{thmTorusForZero}
conclude the proof of Theorem \ref{thmImage}.

\section{Topology near degeneration
locus}\label{secNeighborhoodGeneral}

In this section we study the topology of a small neighborhood of
$\Xn^0$. The space $\Xn$ is a smooth manifold in vicinity of
$\Xn^0$, see Construction \ref{conDegenerationPoints}.

\begin{rem}\label{remFreeOutside}
Note that the $T$-action is free outside $\Xn^0$ and admits a
section given by the formula \eqref{eqRotatedMatrix}. However, the
free part of the action is larger than \edt{the outside of} $\Xn^0$: the action is also
free over the interiors of facets of $\Qn^0\cong \PT^{n-1}$. The
whole free action $\Xn^{\free}\to\Xn^{\free}/T$ does not admit a
section, as explained below.
\end{rem}

Recall that $p\colon \Xn\to \Co$ maps a matrix
$L(\underline{a},\underline{b})$ to the product $B=\prod_1^nb_i$.
For a small~$\varepsilon$ consider the preimage of points close to
zero:
\[
\Xneps=p^{-1}(\{z\in\Co\mid |z|\leqslant \varepsilon\}).
\]
According to Proposition \ref{propPreimageOfNonzero}, $\Xneps$ is
a manifold with boundary, the boundary $\dd\Xneps$ being the subset
\begin{equation}\label{eqEpsTorus}
\Xn^{=\eps}=p^{-1}(\{|z|=\varepsilon\})\cong S_{\eps}^1\times
\T^{n-1}\times T^{n-1}\cong T^{2n-1},
\end{equation}
where $S_{\eps}^1=\{z\mid |z|=\eps\}$, $\T^{n-1}$ is
the Liouville--Arnold torus, and $T^{n-1}$ \edt{is} the acting torus.

It will be useful to incorporate the circle $S_{\eps}^1$
into the action to obtain a $T^n$-action on $\Xneps$.

\begin{con}\label{conModelOfNeighb}
Consider a topological manifold with boundary
$W=\T^{n-1}\times[0,1]$. Its boundary consists of two connected
components
\edt{
\[
\dd W=\dd_0W\sqcup\dd_1W,\qquad
\dd_0W=\T^{n-1}\times\{0\},\quad\dd_1 W=\T^{n-1}\times\{1\}.
\]
}
On the left component $\dd_0W$, we introduce the wonderful cell
structure $\PT^{n-1}$, constructed in Section \ref{secTilings}.
This procedure subdivides $\dd_0W$ into $n$ permutohedra
$\PT_1,\ldots,\PT_n$ of dimension $n-1$ so that every cell of
dimension $k$ lies in $n-k$ top-dimensional cells. This makes $W$
a manifold with corners (understood in a broad topological sense).
We leave the right boundary component $\dd_1 W$ unchanged: no face
structure is imposed on $\dd_1 W$.

Let $T^n=\{t=(t_1,\ldots,t_n)\mid |t_i|=1\}$ be a compact
$n$-torus and $T_I$, $I\subseteq[n]$ be its coordinate subtorus,
\[
T_I=\{t\in T^n\mid t_j=1, j\notin I\}.
\]
Consider the space
\[
Y=W\times T^n/\sim
\]
where $(r,t)$ and $(r',t')$ are identified whenever $r=r'$ lies in
the intersection of facets $\{\PT_i\mid i\in I\}$ and $t^{-1}t'\in T_I$
for some subset $I\subseteq[n]$. This construction can be
considered as \edt{a} particular case of either moment-angle manifold
construction for simplicial posets (see \cite{PL,BPnew}) or the
construction of locally standard actions (see \cite{Yo}). The
space $Y$ is a particular case of the collar models introduced in
\cite{Ay3}.

The space $Y$ is a manifold with boundary $\dd Y=\dd_1 W\times
T^n\cong \T^{n-1}\times T^n$. It carries the action of $T^n$ which
is free on the boundary and its orbit space is $W$.

Consider the induced action of the subtorus
\[
T^{n-1}=\{t_1t_2\cdots t_n=1\}\subset T^n
\]
on the space $Y$. It can be checked \edt{(details can be found in \cite[Construction 5.9]{AyLoc})}
that the orbit space $Y/T^{n-1}$ is homeomorphic to
$\T^{n-1}\times D^2$.
\end{con}

\begin{thm}\label{thmTopologyNeighb}
The space $\Xneps$ is $T^{n-1}$-equivariantly homeomorphic to the
collar model~$Y$.
\end{thm}

\begin{proof}
In \cite{AyLoc} we developed a topological theory of complexity
one torus actions. The main concepts are recalled here. By
definition, an effective action of $T\cong T^{n-1}$ on $X=X^{2n}$
is called a \emph{strictly appropriate action in general position}, if
the following conditions hold.
\begin{enumerate}
\item The action has finitely many fixed points.
\item Stabilizers of all points are connected.
\item Each connected component of each equivariant skeleton $X_j$ contains a fixed point.
\item For every fixed point $x$, the weights
$\alpha_1,\ldots,\alpha_n\in \Hom(T^{n-1},S^1)\cong \Zo^{n-1}$ of
the tangent representation are in general position, which means
that every $n-1$ of them are linearly independent.
\end{enumerate}
For such actions we proved that the orbit space $Q=X/T^{n-1}$ is a \edt{closed}
topological manifold of dimension $n+1$. The orbits of dimensions
less than $n-1$ form a subset $Z\subset Q$ which is called a
\emph{sponge}. A sponge is an $(n-2)$-dimensional subset of $Q$
locally modeled by a $(n-2)$-skeleton of $\Rg^n$. The free part of
action gives the principal $T^{n-1}$-bundle
\[
X^{\free}\to Q\setminus Z.
\]
This bundle is classified by the cohomology class $e\in
H^2(Q\setminus Z;H_1(T^{n-1}))$, which is called the \emph{Euler
class} of the action. Proposition 3.7 of \cite{AyLoc} asserts that
equivariant topological type of $X$ is uniquely determined by the
triple $(Q,Z,e)$ (which essentially means that the information on
stabilizers of the action can be recovered from the class $e$).


\edt{Let $x\in Z$ be an orbit with nontrivial stabilizer, and $U_x$ be a small
open disc neighborhood of $x$ in $Q$.} The inclusion $i_x\colon U_x\to Q$ induces a homomorphism
\[
i_x^*\colon H^2(Q,Q\setminus Z;H_1(T^{n-1}))\to
H^2(U_x,U_x\setminus Z;H_1(T^{n-1})).
\]
The class $e_x=i_x^*(e)\in H^2(U_x,U_x\setminus Z;H_1(T^{n-1}))$
is called the local Euler class at $x$. It was noted in
\cite{AyLoc} that local Euler classes are always nonzero. In
particular, the global Euler class is always non-zero for suitable
actions of complexity one.

These constructions work similarly if $X$ is a manifold with
boundary, and the torus action is free on the boundary. In this
case, $Q=X/T^{n-1}$ is a manifold with boundary $\dd X/T^{n-1}$.
The sponge of the action lies in the interior of $Q$. Under
certain conditions the local Euler classes at fixed points
determine the space $X$ uniquely.

Assume $Q$ has the form $Q_M=M\times D^2$, where $M$ is a closed
$(n-1)$-manifold with a fixed simple cell decomposition. Assume
that the sponge $Z_M$ is the $(n-2)$-skeleton of this cell
structure, and we have
\[
Z_M=M^{(n-2)}=M^{(n-2)}\times\{0\}\subset M\times D^2=Q_M.
\]

\begin{prop}[{\cite[Prop.5.7]{AyLoc}}]\label{propClassificationEarlier}
Let $X$ be a manifold with boundary, which carries a strictly
appropriate torus action in general position such that the orbit
space and the sponge of the action are given by $(Q_M,Z_M)$.
Assume that the free action of $T$ on the boundary is a trivial
principal bundle. Then the local Euler classes at fixed points
uniquely determine the $T^{n-1}$-equivariant homeomorphism type of
$X$.
\end{prop}

Apply this proposition to spaces $\Xneps$ and $Y$. The orbit space
is $\T^{n-1}\times D^2$ in both cases. The sponge of the action is
the $(n-2)$-skeleton of the wonderful cell subdivision
$\PT^{n-1}$, defined earlier. The free action on the boundary is a
trivial principal bundle. This is true for $\Xneps$ since there is
a section of the action, see \edt{Remark} \ref{remFreeOutside}. This is
true for $Y$ since \edt{$Y=W\times T^n/\sim$, and the $T^n$-action over
$\dd_1W$} is a trivial principal bundle.

Finally, consider any fixed point $x=L_\sigma$ of $\Xneps$. The
tangent representation at $x$ is isomorphic to the standard action
of
\[
T^{n-1}=\{t_1\cdots t_n=1\}\subset T^n=\{(t_1,\ldots,t_n)\}
\]
on $\Co^n$ (the infinitesimal action just rotates off-diagonal
entries, so that the angles of rotation sum to zero). \edt{This action
coincides with the
action of $T^{n-1}$ in the neighborhood of the corresponding fixed
point of $Y$, by the definition of $Y$}.
Therefore the local Euler classes of $\Xneps$ and $Y$ coincide at
each fixed point.

Proposition \ref{propClassificationEarlier} then implies the
existence of $T^{n-1}$-homeomorphism $\Xneps\cong Y$.
\end{proof}

\section{Enumerative combinatorics of the wonderful subdivision}\label{secNeighborhoodCombinatorics}

In this section we study the enumerative invariants of the
permutoheral cell complex $\PT^{n-1}$ or, equivalently, its dual
simplicial poset $\KPT^{n-1}$. These invariants will be used
further to describe the homological structure of $\Xn$. At
first, we recall several standard definitions from commutative algebra
and combinatorics.

\begin{defin}\label{definSimpPoset}
A finite partially ordered set $S$ is called \emph{simplicial} if
it has the minimal element $\minel\in S$ and, for any $I\in
S$, the order interval $\{J\in S\mid J\leqslant I\}$ is isomorphic
to the poset of faces of a $k$-dimensional simplex, for some number
$k\geqslant 0$.
\end{defin}

The elements of $S$ are called \emph{simplices}. The number $k$
from the definition is called the dimension of a simplex $I$. A
simplex of dimension $0$ is called a \emph{vertex}. The
geometrical realization of $S$ is the simplicial cell complex,
obtained by gluing geometrical simplices according to the order
relation in $S$, see \cite{BPposets} for details. In the
following we only consider \emph{pure} simplicial posets, which
means that all maximal elements of $S$ have the same dimension. A
simplicial poset is called \emph{a homology sphere} (resp.
\emph{a homology manifold}) if its geometrical realization is a
homology sphere (resp. a homology manifold).

\begin{con}
Let $f_j$ denote the number of $j$-dimensional simplices of $S$
for $j=-1,0,\ldots,n-1$, in particular, $f_{-1}=1$ (the empty
simplex $\minel$ has dimension $-1$). $h$-numbers of $S$ are
defined from the relation:
\begin{equation}\label{eqHvecDefin}
\sum_{j=0}^nh_jt^{n-j}=\sum_{j=0}^nf_{j-1}(t-1)^{n-j},
\end{equation}
where $t$ is a formal variable. Let $\br_j(S)=\dim \Hr_j(S)$ be
the reduced Betti number of the geometric realization of $S$.
\emph{$h'$- and $h''$-numbers} of $S$ are defined as follows
\begin{equation}\label{eqDefHprime}
h_j'=h_j+{n\choose
j}\left(\sum_{s=1}^{j-1}(-1)^{j-s-1}\br_{s-1}(S)\right)\mbox{ for
} 0\leqslant j\leqslant n;
\end{equation}
\begin{equation}\label{eqDefHtwoprimes}
h_j'' = h_j'-{n\choose j}\br_{j-1}(S) = h_j+{n\choose
j}\left(\sum_{s=1}^{j}(-1)^{j-s-1}\br_{s-1}(S)\right)
\end{equation}
for $0\leqslant j\leqslant n-1$, and $h''_n=h'_n$. Sums over
empty sets are assumed zero.
\end{con}

Let $[m]=\{1,\ldots,m\}$ be the vertex set of $S$, $m=f_0$. Let
$R$ be a field or the ring~$\Zo$, and let
$R[m]=R[v_1,\ldots,v_m]$, $\deg v_i=2$, denote the graded
polynomial algebra with $m$ generators, corresponding to the
vertices of $S$. Slightly abusing the terminology, we call the elements of
degree 2 linear, when working with such polynomial rings. For a graded
$R$-module $V^*=\bigoplus_{j=0}^\infty V_j$ we denote by
$\Hilb(V^*;t)$ its Hilbert--Poincare function
$\sum_{j=0}^\infty t^j\rk_RV_j\in \Zo[[t]]$.

\begin{defin}[see \cite{StPosets}]\label{definFaceRing}
The face ring of a simplicial poset $S$ is the commutative
associative graded algebra $R[S]$ over a ring $R$ generated by
formal variables $v_I$, one for each simplex $I\in S$, with
relations
\[
v_{I_1}\cdot v_{I_2}=v_{I_1\cap I_2}\cdot\sum_{J\in I_1\vee
I_2}v_J,\qquad v_{\minel}=1.
\]
Here $I_1\vee I_2$ denotes the set of least upper bounds of
$I_1,I_2\in S$, and $I_1\cap I_2\in S$ is the intersection of
simplices (it is well-defined and unique when $I_1\vee
I_2\neq\minel$). We take the doubled grading on the ring, in which
$v_I$ has degree $2(\dim I+1)$. The natural graded ring
homomorphism $R[m]=R[v_1,\ldots,v_m]\to R[S]$ defines the
structure of the $R[m]$-module on $R[S]$.
\end{defin}

If $R$ is an infinite field, and $\dim S=n-1$, then a generic set of linear
elements $\theta_1,\ldots,\theta_n\in R[S]_2$ is a linear system
of parameters (we remark that linear systems of parameters can be
constructed using characteristic functions on $S$, see e.g.
\cite[Lm.3.5.8]{BPnew}). Let $\Theta$ denote the parametric ideal
of $R[S]$ generated by $\theta_1,\ldots,\theta_n$

\begin{prop}[Reisner, Stanley \cite{Reis,St}]\label{propStanley}
For a pure simplicial poset $S$ of dimension $n-1$ there holds
\[
\Hilb(R[S];t)=\dfrac{h_0+h_1t^2+\cdots+h_nt^n}{(1-t^2)^n}.
\]
For a homology sphere $S$ there holds
$\Hilb(R[S]/\Theta;t)=\sum_ih_it^{2i}$.
\end{prop}

\begin{prop}[Schenzel, Novik--Swartz \cite{Sch,NS,NSgor}]\mbox{}

(1) For a homology manifold $S$ there holds
\[\Hilb(R[S]/\Theta;t)=\sum_ih'_it^{2i}.\]

(2) Let $S$ be a connected $R$-orientable homology manifold of
dimension $n-1$. The $2j$-th graded component of the module
$R[S]/\Theta$ contains a vector subspace
$(I_{NS})_{2j}\cong{n\choose j}\Hr^{j-1}(S;R)$, which is a trivial
$R[m]$-submodule (i.e. $R[m]_+(I_{NS})_{2j}=0$). Let
$I_{NS}=\bigoplus_{j=0}^{n-1}(I_{NS})_{2j}$ be the sum of all
these submodules except the top-degree component. Then the
quotient module $R[S]/\Theta/I_{NS}$ is a Poincare duality
algebra, and there holds
\[\Hilb(R[S]/\Theta/I_{NS};t)=\sum_ih''_it^{2i}.\]
\end{prop}

We now compute the combinatorial characteristics of the simplicial
poset $\KPT^{n-1}$ dual to $\PT^{n-1}$. Combinatorially, the
simplicial cell complex $\KPT^{n-1}$ can be defined as a poset,
whose elements are the faces of the wonderful cell decomposition
$\PT^{n-1}$ and the order is given by the reversed inclusion. It can
be seen that $\KPT^{n-1}$ is a simplicial poset. Recall that
$\stir{n}{k}$ denotes the Stirling number of the second kind, that
is the number of unordered partitions of the set $[n]$ into $k$
nonempty subsets.


\begin{prop}
For the simplicial poset $\KPT^{n-1}$ there holds
\[
f_{k-1}=n(k-1)!\stir{n}{k} \mbox{ for }k=1,2,\ldots,n;\quad
f_{-1}=1;
\]
\begin{equation}\label{eqHnumbers}
h_l=(-1)^l{n\choose n-l}+\sum_{k=1}^l(-1)^{l-k}{n-k\choose
n-l}n(k-1)!\stir{n}{k}\quad\mbox{for }l=0,1,\ldots,n;
\end{equation}
\begin{multline}\label{eqHprimenumbers}
h_l'=(-1)^l{n\choose n-l}+\sum_{k=1}^l(-1)^{l-k}{n-k\choose
n-l}n(k-1)!\stir{n}{k}+\\+{n\choose
l}\sum_{k=2}^{l-1}(-1)^{l-k-1}{n-1\choose k-1} \quad\mbox{ for
}l=0,1,\ldots,n;
\end{multline}
\begin{multline}\label{eqHtwoprimenumbers}
h_l''=(-1)^l{n\choose n-l}+\sum_{k=1}^l(-1)^{l-k}{n-k\choose
n-l}n(k-1)!\stir{n}{k}+\\+{n\choose
l}\sum_{k=2}^{l}(-1)^{l-k-1}{n-1\choose k-1}\quad \mbox{for
}l=0,1,\ldots,n-1,\quad\mbox{ and }h_{n}''=1.
\end{multline}
\end{prop}

\begin{proof}
From the combinatorial description of a permutohedron it follows
that the number $f_{n-k}(\Pe^{n-1})$ is equal to $k!\stir{n}{k}$.
The wonderful subdivision $\PT^{n-1}$ consists of $n$ permutohedra
and each $(n-k)$-dimensional face of $\PT^{n-1}$ lies in exactly
$k$ permutohedral cells, since the subdivision is simple.
Therefore,
\[
f_{k-1}(\KPT^{n-1})=f_{n-k}(\PT^{n-1})=\frac{n}{k}f_{n-k}(\Pe^{n-1})=n(k-1)!\stir{n}{k}
\]
for $k\geqslant 1$. The identity $f_{-1}=1$ holds authomatically.

Since $\KPT^{n-1}$ is a simplicial cell subdivision of the torus
$\T^{n-1}$, we have $\br_j(\KPT^{n-1})={n-1\choose j}$ for
$j\geqslant 1$. Expressions
\eqref{eqHnumbers},\eqref{eqHprimenumbers}, and
\eqref{eqHtwoprimenumbers} follow from the general definitions of $h$-, $h'$-,
and $h''$-numbers.
\end{proof}

\section{Equivariant cohomology}\label{secEquivCohom}

Let $X$ be a $2n$-manifold with a locally standard action of
$T^n$. The orbit space $P=X/T^n$ is a manifold with faces. It means that every codimension $k$
face of $P$ lies in exactly $k$ different facets of $P$. Let $S_P$ denote the simplicial poset
dual to the poset of faces of~$P$. In \cite{AMPZ} we proved

\begin{prop}\label{propAMPZ}
Assume that all proper faces of $P$ are acyclic and the projection
map $X\to P$ admits a section. Then $H^*_{T^n}(X;\Zo)\cong
\Zo[S_P]\oplus H^*(P;\Zo)$ as the rings, and as the modules over
$\Zo[n]\cong H^*(BT^n)$. \edtt{The components of degree 0 are identified
in the direct sum $\Zo[S_P]\oplus H^*(P;\Zo)$}. The ring $H^*(P;\Zo)$ is considered a trivial
$\Zo[n]$-module.
\end{prop}

Now we apply this statement to the space $\Xneps$ which is
$T^{n-1}$-equivariantly homeomorphic to $Y$ (see construction
\ref{conModelOfNeighb}).

\begin{thm}\label{thmEqBettiNeighb}
The Hilbert--Poincare series of the $T^{n-1}$-equivariant
cohomology ring of $\Xneps$ is given by
\[
\Hilb(H^*_{T^{n-1}}(\Xneps);t)=\dfrac{\sum_{i=0}^nh_it^{2i}}{(1-t^2)^{n-1}}+(1+t)^{n}-1-t.
\]
Here $h_i$, the $h$-numbers of the simplicial poset
$\KPT^{n-1}$, are given by \eqref{eqHnumbers}.
\end{thm}

\begin{proof}
Recall that $Y$ is the collar model, that is the locally standard
$T^n$-space over $\T^{n-1}\times [0,1]$. Proposition
\ref{propAMPZ} implies the following isomorphism for the
$T^n$-equivariant cohomology
\[
H^*_{T^n}(Y)\cong\Zo[\KPT^{n-1}]\oplus H^*(\T^{n-1}).
\]
There is an induced action of the $(n-1)$-dimensional subtorus
\begin{equation}\label{eqSubtorusDiag}
T^{n-1}=\{t_1\cdots t_n=1\}
\end{equation}
on $Y$, and Theorem
\ref{thmTopologyNeighb} states that $Y$ and $\Xneps$ are
$T^{n-1}$-equivariantly homeomorphic. To compute the
$T^{n-1}$-equivariant cohomology of $Y$, we first note that there
is a Serre fibration
\[
Y_{T^{n-1}}\stackrel{S^1}{\longrightarrow} Y_{T^{n}}, \qquad
S^1=T^n/T^{n-1},
\]
where $Y_{T^{n-1}}$ and $Y_{T^{n}}$ are the Borel constructions of
$T^{n-1}$- and $T^n$-actions on $Y$ respectively. Consider the
corresponding Serre spectral sequence:
\[
E_2^{p,q}=H_{T^n}^p(Y)\otimes H^q(S^1) \Rightarrow
H^{p+q}_{T^{n-1}}(Y).
\]
The sequence has only two nonzero rows, hence it collapses at the
$E_3$-term. Let $\omega$ denote a generator of $H^1(S^1)$. The second
differential $d_2\colon H^1(S^1)\to H^2_{T^n}(Y)$ of the spectral
sequence coincides with the composition
\[
H^1(T^{n}/T^{n-1})\cong H^2(B(T^{n}/T^{n-1}))\to H^2(BT^n) \to
H^2_{T^n}(Y),
\]
where the middle map is induced by the projection $T^n\to
T^n/T^{n-1}$ and the right map is the defining map for the
$H^*(BT^n)$-module structure on $H^*_{T^n}(Y)$. It follows that
\[
d_2(\omega)=\eta\in \Zo[\KPT^{n-1}]_2\subset
H^2_{T^n}(Y),\quad\mbox{ where }\eta=\sum\nolimits_{i=1}^nv_i,
\]
according to the definition \eqref{eqSubtorusDiag} of the subtorus $T^{n-1}$.
\begin{lem}\label{lemRegElement}
$\eta$ is not a zero divisor in the face ring $\Zo[\KPT^{n-1}]$, or, equivalently,
$\eta$ is a regular element.
\end{lem}

\begin{proof}
We use the standard argument in the theory of face rings. For any non-empty simplex
$I\cong \Delta^{k}$ in $\KPT^{n-1}$ consider the epimorphism \edt{
\[
\varphi_I\colon \ko[\KPT^{n-1}]\to \ko[I]
\]
}
defined by sending $v_J$ to $0$ for all $J\nleqslant I$. Notice that $\ko[I]$ is
just the polynomial algebra in $\dim I+1$ generators. The map $\varphi_I$
is a homomorphism of $\ko[n]$-algebras, with the $\ko[n]$-structure on \edt{$\ko[I]$
defined} by an epimorphism $\psi_I$ sending the excess variables to zeroes.

Assume that there exists $\beta\in \ko[\KPT^{n-1}]$ such that $\eta\cdot\beta=0$. Then \edt{
$\psi_I(\eta)\cdot\varphi_I(\beta)=0$} in~$\ko[I]$. Since there are no zero divisors in $\ko[I]$,
and \edtt{$\psi_I(\eta)=\sum_{i\in I}v_i\neq 0$}, we have $\varphi_I(\beta)=0$ for any simplex $I$ of $\KPT^{n-1}$. The
homomorphism \edt{
\[
\bigoplus\nolimits_I\varphi_I\colon \ko[\KPT^{n-1}]\to \bigoplus\nolimits_I\ko[I]
\]
}
is known to be injective \cite[Thm 3.5.6]{BPnew}. Therefore $\beta=0$.
\end{proof}

According to the lemma, for any nonzero element
$\beta\in \Zo[\KPT^{n-1}]\subset H^*_{T^n}(Y)$, there holds
\edtt{
\[
d_2(\omega \beta)=(d_2\omega)\beta\pm \omega d_2(\beta)=\eta
\beta\neq 0.
\]
}
In other words, $d_2$ is injective on the submodule
$\Zo[\KPT^{n-1}]\otimes H^1(S^1)\subset E_2^{*,1}$.

On the other hand, for any element $\alpha\in H^i(\T^{n-1})\subset
H^i_{T^n}(Y)$, $i>0$, we have
\[
d_2(\omega \alpha)=(d_2\omega)\alpha\pm \omega
d_2(\alpha)=\eta\cdot\alpha=0,
\]
since the products of elements from the components
$H^{+}(\T^{n-1})$ and $\Zo[\KPT^{n-1}]_{+}$ of the ring
$H^*_{T^n}(Y)$ vanish. Therefore the differential $d_2$ vanishes
on \edtt{$H^{+}(\T^{n-1})\otimes H^1(S^1)$}. Finally, we have \edt{
\[
E_3^{p,q}\cong \begin{cases} (\Zo[\KPT^{n-1}]/(\eta)_p)\oplus
H^p(\T^{n-1}),\mbox{ for }q=0;\\
H^p(\T^{n-1}), \mbox{ for }q=1, p>0.
\end{cases}
\]
}
\edt{
Two things should be noted at this place. First, the components of degree zero are identified in $(\Zo[\KPT^{n-1}]/(\eta)_p)\oplus
H^p(\T^{n-1})$. Second, the term $E_3^{0,1}$ vanishes, since the generator $\omega$ of $E_2^{0,1}$ maps to $\eta\neq 0$ by $d_2$.}
The Hilbert--Poincare series of $\Zo[\KPT^{n-1}]$ is given by
$\left(\sum_{i=0}^nh_it^{2i}\right)/(1-t^2)^{n}$ according to
Proposition \ref{propStanley}, \edt{so we have
\[
\Hilb(\Zo[\KPT^{n-1}]/(\eta);t)=\dfrac{\sum_{i=0}^nh_it^{2i}}{(1-t^2)^{n-1}}
\]
by Lemma~\ref{lemRegElement}.}
The statement \edt{of the theorem} now follows from the degeneration of the spectral
sequence at $E_3$-term.
\end{proof}

\begin{cor}\label{corHilbertEqCohom}
For the whole isospectral space $\Xn$ there holds
\[
\Hilb(H^*_{T^{n-1}}(\Xn);t)=\dfrac{\sum_{i=0}^nh_it^{2i}}{(1-t^2)^{n-1}}+R(t),
\]
where $R(t)$ is a polynomial, and $h_i$ are given by
\eqref{eqHnumbers}.
\end{cor}

\begin{proof}
The space $\Xn$ is patched from $\Xneps$ and
$\Xngeps=\{p^{-1}(|z|\geqslant \varepsilon)\}$. Both subsets are
preserved by the torus action, hence the equivariant cohomology
groups can be computed via Mayer--Vietoris exact sequence.
However, the torus action on $\Xngeps$ and $\Xngeps\cap \Xneps$ is
free, so the equivariant cohomology groups of these subsets
coincide with the ordinary cohomology groups of their orbit
spaces. Hence they are concentrated in a finite range of degrees.
The statement now follows from Theorem \ref{thmEqBettiNeighb}.
\end{proof}

\begin{ex}
We check the calculations for the case $n=3$. The isospectral
space $X_{3,\lambda}$ coincides with the
manifold~$\Fl_3$ of complete complex flags. Its equivariant cohomology \edt{is} known to satisfy
\edtt{
\[
\Hilb(H^*_{T^2}(\Fl_3);t)=\dfrac{\Hilb(H^*(\Fl_3);t)}{(1-t^2)^2}=\dfrac{1+2t^2+2t^4+t^6}{(1-t^2)^2}.
\]
}
Using formulas \eqref{eqHnumbers} compute the h-numbers of
$\KPT^2$: $(h_0,h_1,h_2,h_3)=(1,0,6,-1)$. There holds
\[
\dfrac{1+2t^2+2t^4+t^6}{(1-t^2)^2}=\dfrac{1+6t^4-t^6}{(1-t^2)^2}+2t^2,
\]
which confirms Corollary \ref{corHilbertEqCohom}.
\end{ex}

\section{Betti numbers}\label{secBetti}

In this section we describe the additive structure of the cohomology
ring of $\Xn$. As a first step, we compute the homological structure of
the subset $\Xneps$, containing the essential information on the
torus action. It is assumed in this section, that all coefficients are
taken in a fixed field. The next proposition follows from the general technique developed
in \cite{Ay1}.

\begin{prop}\label{propBettiNeighborhood}
The homology modules of $Y\cong \Xneps$ admit the double grading:
$H_j(Y)\cong \bigoplus_{p+q=j} H_{p,q}(Y)$. There holds
\begin{enumerate}
\item $H_{p,q}(Y)\cong H_{p}(\T^{n-1})\otimes
H_q(T^n)$ for $q<p<n$.
\item $H_{p,q}(Y)=0$ for $q>p$.
\item The dimension of $H_{p,p}(Y)$ equals
\[
h_p+{n\choose p}\sum_{k=2}^{p+1}(-1)^{p+k-1}{n-1\choose k-1}
\]
for $p=0,1,\ldots,n-1$. In particular, for $p\geqslant 2$ there
holds $\dim H_{p,p}(Y)=h_p''+{n\choose p}{n-1\choose p}$.
\end{enumerate}
The inclusion map $i\colon\T^{n-1}\times T^n\cong \dd Y\to Y$
induces the homomorphism in homology, which respects the double
grading:
\[
i_*\colon H_p(\T^{n-1})\otimes H_q(T^n)\to H_{p,q}(Y).
\]
This homomorphism is an isomorphism for $q<p$, injective for
$q=p$, and zero for $q>p$.
\end{prop}


Note that the subspace $\Xneps$ does not depend on the parameters $n_+,n_-$
discussed in the previous sections. Now we are in position to compute the Betti numbers
of $\Xn$. Homology of $\Xn$ will certainly depend on parameters $n_+,n_-$, which
encode ``the degree of degeneration'' of this space.

\begin{thm}\label{thmBettiTotal}
The Hilbert--Poincare series for homology of $\Xn$ is given by the formula \edt{
\begin{equation}\label{eqManySummands}
\sum_{i=0}^{2n}\beta_i(\Xn)t^i=H^{\geqslant\eps}(t)+H^{\leqslant\eps}(t)-H^{=\eps}(t)+(1+t)\cdot H^{\Ker}(t),
\end{equation}
}
where
\begin{align}
H^{=\eps}(t) &= (1+t)^{2n-1},\label{eqHomolEps} \\
H^{\geqslant\eps}(t) &= (1+t)^{2n-n_+-n_--2}(1-t+t(1+t)^{n_+}+t(1+t)^{n_-}), \label{eqHomolGEps}\\
H^{\leqslant\eps}(t) &= \sum_{p=0}^{n-1}\left(h_p+{n\choose p}\sum_{k=2}^{p+1}(-1)^{p+k-1}{n-1\choose k-1}\right)t^{2p} +\sum_{q<p<n}{n-1\choose p}{n\choose q}t^{p+q}, \label{eqHomolLEps}\\
H^{\Ker}(t) &= \sum_{(p,e,q,s,r)\in \Upsilon}{n-1-n_+-n_-\choose p}{1\choose e}{n_+\choose q}{n_-\choose s}{n-1\choose r}t^{p+e+q+s+r},\label{eqHomolKer}
\end{align}
The indexing subset $\Upsilon$ in the last expression is defined by the conditions
\begin{equation}\label{eqConditions}
\begin{array}{c}
0\leqslant p\leqslant n-1-n_+-n_-;\quad 0\leqslant e\leqslant 1;\\
0\leqslant q\leqslant n_+;\quad 0\leqslant s \leqslant n_-;\quad 0\leqslant r \leqslant n-1; \\
r+e>p+q+s;\\
\mbox{either }(e=0 \mbox{ and } q+s>0) \mbox{ or }(e=1 \mbox{ and } q>0 \mbox{ and }s>0).
\end{array}
\end{equation}
The $h$-numbers are given by \eqref{eqHnumbers}.
\end{thm}

\begin{proof}
Althouth the result looks awkward, the idea behind this calculation is straightforward: we analyze the
Mayer--Vietoris sequence for the union $\Xn=\Xneps\cup\Xngeps$. Let
$\Xn^{=\eps}(t)$ denote the intersection $\Xneps\cap\Xngeps$. Then there is
a long exact sequence
\edt{
\begin{equation}\label{eqMVhomology}
\to H_i(\Xn^{=\eps})\stackrel{\iota_i}{\to} H_i(\Xneps)\oplus H_i(\Xngeps)\to
H_i(\Xn)\to H_{i-1}(\Xn^{=\eps})\stackrel{\iota_{i-1}}{\to} H_{i-1}(\Xneps)\oplus H_{i-1}(\Xngeps)\to
\end{equation}
}
\edt{Note} that given an exact sequence
\[
A_i\stackrel{\alpha_i}{\to} B_i\to C_i\to A_{i-1}\stackrel{\alpha_{i-1}}{\to} B_{i-1},
\]
the dimension of the vector space in the middle is given by
\edtt{
\begin{equation}\label{eqExactGen}
\dim C_i=\dim B_i-\dim A_i+\dim\Ker\alpha_i+\dim\Ker\alpha_{i-1}.
\end{equation}
}
After introducing a natural notation
\begin{align*}
H^{=\eps}(t) &= \sum\nolimits_{i}\dim H_i(\Xn^{=\eps})t^i, \\
H^{\geqslant\eps}(t) &= \sum\nolimits_{i}\dim H_i(\Xngeps)t^i, \\
H^{\leqslant\eps}(t) &= \sum\nolimits_{i}\dim H_i(\Xneps)t^i, \\
H^{\Ker}(t) &= \sum\nolimits_{i}\dim(\Ker \iota_i\colon H_i(\Xn^{=\eps})\to H_i(\Xneps)\oplus H_i(\Xngeps))t^i.
\end{align*}
formula \eqref{eqExactGen} implies \eqref{eqManySummands}. We need to check formulas \eqref{eqHomolEps}--\eqref{eqHomolKer}.

(1) According to \eqref{eqEpsTorus}, $\Xn^{=\eps}\cong T^{2n-1}$, which implies \eqref{eqHomolEps}.

(2) The space $\Xngeps$ supports a free action of $T^{n-1}$, which admits a section.
Therefore,
\begin{equation}\label{eqProduct}
\Xngeps=\Qngeps\times T^{n-1},
\end{equation}
where $\Qngeps=\Xngeps/T^{n-1}$. The structure of the orbit space
was described in detail in Section~\ref{secOrbitSpace}: $\Qngeps=\Xngeps/T^{n-1}$ is fibered by
Liouville--Arnold tori over $\B\setminus \{|z|<\varepsilon\}$, --- the biangle with a hole.
The homotopy type of the latter space is
\begin{equation}\label{eqHomotHoledFamily}
\Qngeps\simeq \T^{n-n_+-n_--1}\times (S^1\vee\Sigma\T^{n_+}\vee \Sigma\T^{n_-}).
\end{equation}
\edt{Indeed, the space $\Qngeps$ can be collapsed to the subspace $\po^{-1}(\dd\B)$. The latter space
is foliated over the circle $\dd\B=F_+\cup F_-$ by tori: the preimages of points in the interior of $F_+$ are $\T^{n-n_+-n_--1}\times \T^{n_-}$, the preimages of points in the interior of $F_-$ are $\T^{n-n_+-n_--1}\times \T^{n_+}$, and the preimages of points in $F_+\cap F_-$ are the tori $\T^{n-n_+-n_--1}$. It can be seen that $\T^{n-n_+-n_--1}$ splits as a direct factor. The remaining space is the suspension over the disjoint union $\T^{n_+}\sqcup \T^{n_-}$. By a standard property of the suspension, we have $\Sigma(\T^{n_+}\sqcup \T^{n_-})\simeq (S^1\vee\Sigma\T^{n_+}\vee \Sigma\T^{n_-})$, which proves \eqref{eqHomotHoledFamily}.}

Formula \eqref{eqHomolGEps} follows from \eqref{eqProduct} and \eqref{eqHomotHoledFamily}.

(3) Homology of $\Xneps$ are described by Theorem \ref{propBettiNeighborhood}. Formula \eqref{eqHomolLEps}
is its simple consequence.

(4) To derive \eqref{eqHomolKer}, we need to count the dimension of the
vector space of all homology cycles of $\Xn^{=\eps}$, annihilated by both maps
\begin{align*}
  \iota_*^{\leqslant}\colon& H_*(\Xn^{=\eps})\to H_*(\Xneps), \\
  \iota_*^{\geqslant}\colon& H_*(\Xn^{=\eps})\to H_*(\Xngeps).
\end{align*}
The torus $\Xn^{=\eps}$ decomposes into the product
\[
\Xn^{=\eps}\cong S_\eps^1\times \T^{n-1}\times T^{n-1}\cong \T^{n-1-n_+-n_-}\times S_\eps^1\times \T^{n_+}\times \T^{n_-}\times T^{n-1},
\]
where
\begin{enumerate}
  \item the component $T^{n-1}$ corresponds to the acting torus;
  \item $\T^{n_+}$, $\T^{n_-}$ are the components of Liouville--Arnold tori, collapsing over the sides of the biangle $\B$;
  \item $\T^{n-1-n_+-n_-}$
is the surviving component of the Liouville--Arnold torus.
  \item $S^1$ is the circle $\{|z|=\varepsilon\}$, lifted to the total space.
\end{enumerate}
These five components of the torus explain the appearance \edt{of} the five-element indexing subset in the statement.
Let $\theta=\omega_p\otimes\alpha_e\otimes \omega_q^+\otimes \omega_s^-\otimes \nu_r$ be a homology cycle
from
\[
H_*(\Xn^{=\eps})\cong H_*(\T^{n-1-n_+-n_-})\otimes H_*(S^1)\otimes H_*(\T^{n_+})\otimes H_*(\T^{n_-})\otimes H_*(T^{n-1}),
\]
where the degrees of the factors are indicated in their subscripts.

\edt{Lines 1-2 of condition \eqref{eqConditions}
simply indicate the range of the indices: from zero to the dimension of the corresponding torus. We have $\iota_*^{\leqslant}(\theta)=0$ if and only if $r+e>p+q+s$ according to the last part of Proposition~\ref{propBettiNeighborhood}. This explains line 3 of condition \eqref{eqConditions}. Finally, we investigate the map $\iota_*^{\geqslant}\colon H_*(\Xn^{=\eps})\to H_*(\Xngeps)$ and describe its kernel. Note that the tori $T^{n-1}$ and $\T^{n-1-n_+-n_-}$ split as direct factors in the map $\iota^{\geqslant}\colon \Xn^{=\eps}\to \Xngeps$. Hence we have
\[
\Ker \iota_*^{\geqslant}\cong H_*(T^{n-1}\times\T^{n-1-n_+-n_-})\otimes\Ker \tilde{\iota}_*^{\geqslant},
\]
where $\tilde{\iota}^{\geqslant}$ is the inclusion map from $\Qn^{=\eps}/\T^{n-1-n_+-n_-}\cong S^1\times\T^{n_+}\times\T^{n_-}$ to $\Qngeps/\T^{n-1-n_+-n_-}\simeq \Sigma(\T^{n_+}\sqcup\T^{n_-})\simeq S^1\vee \Sigma\T^{n_+}\vee\Sigma\T^{n_-}$. We have
\begin{equation}\label{eqCollapsingFormulae}
\tilde{\iota}_*^{\geqslant}(\alpha_e\otimes\omega_q^+\otimes\omega_s^-)=
\begin{cases}
  1, & \mbox{if } e=0, q+s=0, \\
  0+0+0, & \mbox{if } e=0, q+s>0, \\
  \alpha_e+0+0, & \mbox{if } e=1, q+s=0,\\
  0+\Sigma\omega_q^++0, & \mbox{if } e=1, q>0, s=0, \\
  0+0+\Sigma_s\omega^-, & \mbox{if } e=1, q=0, s>0,\\
  0+0+0, & \mbox{if } e=1, q>0, s>0.
\end{cases}
\end{equation}
A convenient way to see these formulae is to replace $\tilde{\iota}^{\geqslant}$ by a homotopy equivalent map illustrated on Fig.\ref{figSuspensions}. For example, a homology element $\alpha_e\otimes\omega_q^+\otimes\omega_s^-$ with $e=1$, $q>0$, $s=0$ is represented by a subtorus $S^1\times \T^{q}\times\pt$. The map $\tilde{\iota}_1^{\geqslant}$ sends this cycle to $\Sigma \T^q$ in the right half of $\Sigma(\T^{n_+}\times\T^{n_-}\sqcup \T^{n_+}\times\T^{n_-})$. The left half vanishes since $\T^{n_+}$ is collapsed, together with $\T^q$ over the left part. The same reasoning explains all other formulae~\eqref{eqCollapsingFormulae}}\edtt{.}

\begin{figure}[h]
\begin{center}

\begin{tikzpicture}[scale=0.7]

\filldraw[fill=black!10, draw=black] (0,0) circle (2cm);
\filldraw[fill=white, draw=black] (0,0) circle (1cm);
\filldraw[color=black!25] (1.5,0) ellipse (0.5cm and 0.25cm);
\filldraw[color=black!25] (-1.5,0) ellipse (0.5cm and 0.25cm);
\node at (1.5,-0.7) {$\T^{n_+}\times\T^{n_-}$};
\node at (0,-2.5) {$S^1\times\T^{n_+}\times\T^{n_-}$};

\draw[->] (2.5,0) -- (4.5,0);
\node at (3.5,0.5) {$\tilde{\iota}_1^{\geqslant}$};

\filldraw[fill=black!10, draw=black] (7,0) circle (2cm);
\filldraw[fill=white, draw=black] (7,0) ellipse (1cm and 2cm);
\filldraw[color=black!25] (8.5,0) ellipse (0.5cm and 0.25cm);
\filldraw[color=black!25] (5.5,0) ellipse (0.5cm and 0.25cm);
\node at (8.5,-0.7) {$\T^{n_+}\times\T^{n_-}$};
\node at (5.5,-0.7) {$\T^{n_+}\times\T^{n_-}$};
\node at (7,-2.5) {$\Sigma(\T^{n_+}\times\T^{n_-}\sqcup \T^{n_+}\times\T^{n_-})$};

\draw[->] (9.5,0) -- (11.5,0);
\node at (10.5,0.5) {$\tilde{\iota}_2^{\geqslant}$};

\filldraw[fill=black!10, draw=black] (14,0) circle (2cm);
\filldraw[fill=white, draw=black] (14,0) ellipse (1cm and 2cm);
\filldraw[color=black!25] (15.5,0) ellipse (0.5cm and 0.25cm);
\filldraw[color=black!25] (12.5,0) ellipse (0.5cm and 0.25cm);
\node at (15.5,-0.7) {$\T^{n_+}$};
\node at (12.5,-0.7) {$\T^{n_-}$};
\node at (14,-2.5) {$S^1\vee\Sigma\T^{n_+}\vee\Sigma\T^{n_-}$};

\end{tikzpicture}

\end{center}
\caption{\edt{The map $\tilde{\iota}^{\geqslant}$ pinches the product $S^1\times(\T^{n_+}\times\T^{n_-})$ at two points of $S^1$, then collapses $\T^{n_+}$ on the left side and $\T^{n_-}$ on the right side.}}\label{figSuspensions}
\end{figure}

This calculation shows that $\theta\in\Ker \iota_*^{\geqslant}$ if and only if the 4-th line of condition \eqref{eqConditions} holds. Therefore, both maps $\iota_*^{\leqslant}$ and $\iota_*^{\geqslant}$ annihilate $\theta$ if and only if the 5-tuple
$(p,e,q,s,r)$ satisfies the conditions \eqref{eqConditions}. This proves \eqref{eqHomolKer}.
\end{proof}

\begin{ex}
The Betti numbers computed for small values of $n$ are shown in the Tables
\ref{tableManifoldBetti} and \ref{tableDegenBetti}.
\end{ex}

\begin{table}[h]
  \centering
\begin{tabular}{c||c|c|c|c|c|c|c|c|c|c|c|c|c|}
n&$\beta_0$&$\beta_1$&$\beta_2$&$\beta_3$&$\beta_4$&$\beta_5$&$\beta_6$&$\beta_7$&$\beta_8$&$\beta_9$&$\beta_{10}$&$\beta_{11}$&$\beta_{12}$\\
\hline
\hline
3 & 1 & 0 & 2 & 0 & 2 & 0 & 1 &&&&&&\\
\hline
4 & 1 & 1 & 6 & 2 & 16 & 2 & 6 & 1 & 1 &&&&\\
\hline
5 & 1 & 2 & 13 & 9 & 65 & 16 & 65 & 9 & 13 & 2 & 1 &&\\
\hline
6 & 1 & 3 & 23 & 25 & 203 & 67 & 456 & 67 & 203 & 25 & 23 & 3 & 1\\
\hline
\end{tabular}
  \caption{Manifold case, $n_+=1, n_-=1$.}\label{tableManifoldBetti}
\end{table}

\begin{table}[h]
  \centering
\begin{tabular}{c||c|c|c|c|c|c|c|c|c|c|c|c|c|}
n &$\beta_0$&$\beta_1$&$\beta_2$&$\beta_3$&$\beta_4$&$\beta_5$&$\beta_6$&$\beta_7$&$\beta_8$&$\beta_9$&$\beta_{10}$&$\beta_{11}$&$\beta_{12}$\\
\hline
\hline
4 & 1 & 0 & 3 & 1 & 16 & 3 & 9 & 2 & 1 &&&&\\
\hline
5 & 1 & 0 & 4 & 2 & 57 & 16 & 77 & 22 & 24 & 4 & 1&&\\
\hline
6 & 1 & 0 & 5 & 4 & 167 & 55 & 471 & 115 & 276 & 61 & 39 & 5 & 1\\
\hline
\end{tabular}
  \caption{Most degenerate case, $n_+,n_-$ are the maximal possible for the given $n$, that is $n_++n_-=n-1$.}\label{tableDegenBetti}
\end{table}

\begin{cor}\label{corOddDeg}
The space $\Xn$ has nonzero Betti numbers in odd degrees for $n\geqslant 4$.
\end{cor}

\begin{proof}
Theorem \ref{thmBettiTotal} implies that $\beta_1(\Xn)=n-1-n_+-n_-$.
This number is nonzero unless $n_+$, $n_-$ are the maximal possible,
representing the most degenerate case. For the most degenerate case, and $n\geqslant 4$,
$\beta_{2n-1}(\Xn)$ is nonzero. To prove this, it is sufficient to estimate the
coefficient at $t^{2n-1}$ in \eqref{eqHomolKer} by $2$ from below. Indeed, the only
negative term in \eqref{eqManySummands} comes from $H^{\Ker}(t)$, and its coefficient at
$t^{2n-1}$ is one. The estimation is straightforward.
\end{proof}

Recall that with any action of a torus $T$ on a space $X$ one can associate
a fibration $r\colon X\times_TET\stackrel{X}{\to}BT$,
where $ET$ is a contractible space, carrying the free action of $T$ and
$BT$ is the classifying space of a torus. It can be assumed that $ET=(S^\infty)^k$
and $BT=(\CP^\infty)^k$, where $k=\dim T$.

\begin{defin}\label{definEqForm}
The space $X$ is called \emph{equivariantly
formal} in the sense of \edt{Goresky--Kottwitz--Macpherson} \cite{GKM} if the Serre spectral sequence
\begin{equation}\label{eqSerreSpecSeq}
E_2^{*,*}\cong H^*(BT)\otimes H^*(X)\Rightarrow H^*(X\times_TET)=H^*_T(X),
\end{equation}
i.e. the spectral sequence of the fibration $r$, degenerates at its second page.
\end{defin}

\begin{prop}
The space $\Xn$ is not equivariantly formal for $n\geqslant 4$.
\end{prop}

\begin{proof}
Assume that $\Xn$ is equivariantly formal. The degeneration of the Serre
spectral sequence \eqref{eqSerreSpecSeq} at a second page then implies
\edt{
\begin{equation*}
\Hilb(H^*_T(\Xn);t)=\Hilb(H^*(\Xn);t)\cdot \Hilb(H^*(BT^{n-1});t)=\Hilb(H^*(\Xn);t)\cdot (1-t^2)^{-(n-1)}.
\end{equation*}
}
This identity and Corollary \ref{corOddDeg} imply that
$H^*_T(\Xn)$ has nontrivial components of arbitrarily high odd degree. This contradicts
to Corollary \ref{corHilbertEqCohom}.
\end{proof}

\begin{rem}
The fundamental group of $\Xn$ can be explicitly described as well.
As in the computations above, we consider the decomposition $\Xn=\Xneps\cup\Xngeps$
and apply van Kampen theorem. For the intersection there holds
\edtt{
\[
\pi_1(\Xn^{=\eps})=\pi_1(T^{n_+}\times T^{n_-}\times T^{n-1-n_+-n_-}\times S_\eps^1\times T^{n-1})=\Zo^{n_+}\oplus \Zo^{n_-}\oplus \Zo^{n-1-n_+-n_-}\oplus \Zo^{n}.
\]
}
The summands $\Zo^{n_+}$ and $\Zo^{n_-}$ vanish in $\pi_1(\Xngeps)$ since
the corresponding components of Lioville--Arnold tori are collapsed over the facets of biangle $\B$.
The summand $\Zo^{n}=\pi_1(S_\eps^1\times T^{n-1})$ vanishes in $\pi_1(\Xneps)$ according
to \cite{Zeng}. The result of this paper asserts that for a locally standard
action of $T$, $\dim T=n$, on $M$, \edt{$\dim M=2n$}, having a fixed point, there holds $\pi_1(M)\cong \pi_1(M/T)$.
In other words, any loop on the acting torus can be contracted via a fixed point.
This result is applied to $\Xneps\cong Y$, carrying the extended action of $T^n$.

No other loops appear in $\Xngeps$ and $\Xneps$, therefore
\[
\pi_1(\Xn)\cong \Zo^{n-1-n_+-n_-}.
\]
Hence $\Xn$ is simply connected if and only if the spectrum $\lambda$ satisfies $n_-+n_+=n-1$.
This corresponds to the most degenerate situation, considered in Remark \ref{remDegenerateCheb}.
\end{rem}

\section*{Acknowledgements}\label{secThanks}
The author would like to thank Victor Buchstaber and Igor Krichever for
their valuable comments on this work. I am also grateful to Dmitry Gugnin and Alexander Gaifullin
who brought the subject of crystallizations and the paper \cite{FGG} to my attention. \edt{I thank the anonymous referee for
numerous valuable comments and suggestions which helped to improve the exposition of the paper.}


\begin{thebibliography}{99}


\bibitem{Ay1} A.\,Ayzenberg, \textit{Locally standard torus actions and h'-vectors of simplicial
posets}, J. Math. Soc. Japan 68:4 (2016), 1--21 (preprint:
\href{http://arxiv.org/abs/1501.07016}{arXiv:1501.07016}).

\bibitem{Ay2} A.\,Ayzenberg, \textit{Homology cycles in manifolds with locally standard torus
actions}, Homology, Homotopy Appl. 18:1 (2016), 1--23 (preprint:
\href{http://arxiv.org/abs/1502.01130}{arXiv:1502.01130}).

\bibitem{Ay3} A.\,Ayzenberg, \textit{Topological model for h''-vectors of simplicial manifolds}, Bol. Soc. Mat.
Mexicana (2016), 1--9 (preprint:
\href{http://arxiv.org/abs/1502.05499}{arXiv:1502.05499}).

\bibitem{AyLoc} A.\,Ayzenberg, \textit{Torus actions of complexity one and their local properties}, preprint:
\href{https://arxiv.org/abs/1802.08828}{arXiv:1802.08828}).

\bibitem{ABhess} A.\,A.\,Ayzenberg, V.\,M.\,Buchstaber, \textit{Manifolds of isospectral matrices and Hessenberg
varieties}, preprint:
\href{https://arxiv.org/abs/1803.01132}{arXiv:1803.01132}

\bibitem{ABarrow} A.\,A.\,Ayzenberg, V.\,M.\,Buchstaber, \textit{Manifolds of isospectral arrow matrices}, preprint:
\href{https://arxiv.org/abs/1803.10449}{arXiv:1803.10449}

\bibitem{AMPZ} A.\,Ayzenberg, M.\,Masuda, S.\,Park, H.\,Zeng, \textit{Cohomology of toric
origami manifolds with acyclic proper faces}, J. Symplectic Geom.
Vol. 15:3 (2017), 645--685 (preprint:
\href{http://arxiv.org/abs/1407.0764}{arXiv:1407.0764}).

\bibitem{BFR} A.\,M.\,Bloch, H.\,Flaschka, T.\,Ratiu, \textit{A convexity theorem
for isospectral manifolds of Jacobi matrices in a compact Lie
algebra}, Duke Math.J., Vol. 61:1 (1990), 41--65.

\bibitem{BPposets} V.\,M.\,Buchstaber, T.\,E.\,Panov, \textit{Combinatorics of Simplicial Cell Complexes and Torus Actions}, Proc. Steklov Inst. Math. 247 (2004), 1--17.

\bibitem{BPnew} V.\,Buchstaber, T.\,Panov, Toric Topology, Math. Surveys Monogr., 204, AMS, Providence, RI, 2015.

\bibitem{BTober} V.\,M.\,Buchstaber (joint with S.\,Terzi\'{c}), \textit{(2n, k)-manifolds and applications}, Mathematisches
Forschung Institut Oberwolfach, Report No. 27/2014, p. 58, DOI: 10.4171/OWR/2014/27

\bibitem{BT} V.\,M.\,Buchstaber, S.\,Terzi\'{c}, \textit{Topology and geometry of the canonical action of $T^4$ on the complex
Grassmannian $G_{4,2}$ and the complex projective space $\CP^5$}, Moscow Mathematical Journal, 16:2 (2016), 237--273 (preprint:
\href{https://arxiv.org/abs/1410.2482}{arXiv:1410.2482}).

\bibitem{BT2} V.\,M.\,Buchstaber, S.\,Terzi\'{c}, \textit{Toric topology of the complex Grassmann manifolds}, preprint:
\href{https://arxiv.org/abs/1802.06449}{arXiv:1802.06449}

\bibitem{ChShS} Yu.\,B.\,Chernyakov, G.\,I.\,Sharygin, A.\,S.\,Sorin, \textit{Bruhat Order in Full Symmetric Toda System},
Communications in Mathematical Physics, Vol.330:1 (2014), 367--399.

\bibitem{ConSloane} J.\,H.\,Conway, N.\,J.\,A.\,Sloane, Sphere Packings, Lattices, and Groups, third edition, A Series of Comprehensive Studies in Mathematics, V.290, Springer, 1998.

\bibitem{DJ} M.\,Davis, T.\,Januszkiewicz, \textit{Convex polytopes, Coxeter orbifolds and torus actions}, Duke Math. J. 62:2 (1991), 417--451.


\bibitem{DNT} P.\,Deift, T.\,Nanda, C.\,Tomei, \textit{Ordinary differential equations and the symmetric eigenvalue problem}, SIAM J. Numer.Anal. 20:1 (1983), 1--22.

\bibitem{FGG} M.\,Ferri, C.\,Gagliardi, L.\,Grasselli, \textit{A graph-theoretical representation of PL-manifolds ---
A survey on crystallizations}, Aequationes Mathematicae 31 (1986), 121--141.

\bibitem{GKM} M.\,Goresky, R.\,Kottwitz, R.\,MacPherson, \textit{Equivariant
cohomology, Koszul duality, and the localization theorem}, Invent.
math. 131, 25--83 (1998).

\bibitem{KrichOld} I.\,Krichever, \textit{Nonlinear equations and elliptic curves}, Journal of Soviet
Mathematics 28:1 (1985), 51--90.

\bibitem{Krich} I.\,Krichever, K.\,L.\,Vaninsky, \textit{The periodic and open Toda
lattice}, Mirror Symmetry IV, AMS/IP. vol 33, 139-158 (2002),
(preprint:
\href{https://arxiv.org/abs/hep-th/0010184}{arXiv:hep-th/0010184}).

\bibitem{Kur} S.\,Kuroki, \textit{Introduction to GKM-theory}, Trends in Mathematics - New Series
11:2 (2009), 111--126.

\bibitem{PL} Z.\,L\"{u}, T.\,Panov, \textit{Moment-angle complexes from simplicial
posets}, Cent. Eur. J. Math. 9:4 (2011), 715--730.

\bibitem{VanM} P.\,van Moerbeke, \textit{The Spectrum of Jacobi
Matrices}, Inventiones math. 37 (1976), 45--81.

\bibitem{Moser} J.\,Moser, \textit{Finitely many points on the line under the influence of
an exponential potential --- an integrable system}, in Dynamical
Systems, Theory and Applications, Lecture Notes in Physics, 38,
Springer-Verlag, Berlin, 1975, 467--497.

\bibitem{NS} I.\,Novik, E.\,Swartz, \textit{Socles of Buchsbaum modules,
complexes and posets}, Adv. Math. 222, 2059--2084 (2009)
(preprint:
\href{https://arxiv.org/abs/0711.0783}{arXiv:0711.0783}).

\bibitem{NSgor} I.\,Novik, E.\,Swartz, \textit{Gorenstein rings through face rings of
manifolds}, Composit. Math. 145 (2009), p.993--1000 (preprint:
\href{https://arxiv.org/abs/0806.1017}{arXiv:0806.1017}).

\bibitem{Panina} G.\,Panina, \textit{Cyclopermutohedron}, Proceedings of the Steklov Institute of Mathematics,
Vol. 288 (2015), 132--144 (preprint:
\href{https://arxiv.org/abs/1401.7476}{arXiv:1401.7476}).

\bibitem{Reis} G.\,Reisner, \textit{Cohen-Macaulay quotients of polynomial rings}, Advances in Math. V.21, 1 (1976), p.30--49.

\bibitem{Sch} P.\,Schenzel, \textit{On the Number of Faces of Simplicial Complexes
and the Purity of Frobenius}, Math. Zeitschrift 178 (1981),
p.125--142.

\bibitem{St} R.\,Stanley, \textit{Combinatorics and Commutative Algebra},
Progress in Mathematics 41.
Inc., 1996.

\bibitem{StPosets} R.\,Stanley, \textit{f-vectors and h-vectors of simplicial posets}, J. Pure
Appl. Algebra 71 (1991), 319--331.

\bibitem{Tomei} C.\,Tomei, \textit{The topology of isospectral manifolds of tridiagonal matrices},
Duke Math.Journal (1984), Vol. 51:4.

\bibitem{TomeiON} C.\,Tomei, \textit{The Toda lattice, old and new} The Journal of Geometric Mechanics
5(4) (preprint \href{https://arxiv.org/abs/1508.03229v1}{
arXiv:1508.03229}).

\bibitem{Yo} T.\,Yoshida, \textit{Local torus actions modeled on the standard
representation}, Adv. Math. 227 (2011), pp. 1914--1955 (preprint
\href{https://arxiv.org/abs/0710.2166}{arXiv:0710.2166}).

\bibitem{Zeng} H.\,Zeng, \textit{The fundamental group of locally standard $T$-manifolds}, Algebraic and Geometric Topology 18 (2018), pp. 3031--3035.

\end{thebibliography}
\end{document}